\documentclass[12pt,a4paper]{amsart}
\usepackage[marginratio=1:1]{geometry}

\usepackage[T1]{fontenc}
\usepackage[utf8]{inputenc}
\usepackage[american]{babel}
\usepackage[draft=false]{hyperref}
\usepackage{enumitem}
\usepackage{graphicx}

\hypersetup{
	colorlinks,
	linkcolor={red!50!black},
	citecolor={blue!50!black},
	urlcolor={blue!80!black}
}

\usepackage{amsfonts}
\usepackage{amssymb}
\usepackage{tikz-cd}
\usepackage{float}
\usepackage{mathtools}
\makeatletter
\providecommand*{\twoheadrightarrowfill@}{%
	\arrowfill@\relbar\relbar\twoheadrightarrow
}
\providecommand*{\twoheadleftarrowfill@}{%
	\arrowfill@\twoheadleftarrow\relbar\relbar
}
\providecommand*{\xtwoheadrightarrow}[2][]{%
	\ext@arrow 0579\twoheadrightarrowfill@{#1}{#2}%
}
\providecommand*{\xtwoheadleftarrow}[2][]{%
	\ext@arrow 5097\twoheadleftarrowfill@{#1}{#2}%
}
\makeatother

\newcommand{\C}{{\mathfrak C}}

\newcommand{\R}{{\mathbb{R}}}

\newcommand{\fcolon}{\colon}

\DeclareMathOperator{\tp}{{tp}}
\DeclareMathOperator{\Th}{{Th}}
\DeclareMathOperator{\gal}{{Gal}}
\DeclareMathOperator{\cl}{{cl}}

\DeclareMathOperator{\id}{{id}}
\DeclareMathOperator{\aut}{{Aut}}
\DeclareMathOperator{\autf}{{Autf}}

\DeclareMathOperator{\topo}{{top}}
\DeclareMathOperator{\defi}{{def}}
\DeclareMathOperator{\ext}{{ext}}

\DeclareMathOperator{\Sl}{{SL}}

\DeclareMathOperator{\st}{{st}}
\DeclareMathOperator{\bdd}{{bdd}}
\DeclareMathOperator{\acl}{{acl}}
\DeclareMathOperator{\Fix}{{Fix}}
\DeclareMathOperator{\Mod}{{Mod}}

\DeclarePairedDelimiter\floor{\lfloor}{\rfloor}

\newtheorem{thm}{Theorem}[section]
\newtheorem{conj}[thm]{Conjecture}
\newtheorem{ques}[thm]{Question}

\newtheorem{lem}[thm]{Lemma}
\newtheorem{fct}[thm]{Fact}
\newtheorem{cor}[thm]{Corollary}
\newtheorem{prop}[thm]{Proposition}

\theoremstyle{remark}
\newtheorem{rem}[thm]{Remark}
\theoremstyle{definition}
\newtheorem{dfn}[thm]{Definition}
\newtheorem*{sbclm}{Subclaim}
\newtheorem{clm*}{Claim}
\newtheorem{ex}[thm]{Example}
\newcounter{claimcounter}[thm]

\title{Amenability, definable groups, and automorphism groups}
\author{Krzysztof Krupi\'nski}
\email[K.\ Krupi\'{n}ski]{kkrup@math.uni.wroc.pl}
\address[K.\ Krupi\'nski]{
	Instytut Matematyczny, Uniwersytet Wrocławski\\
	pl. Grunwaldzki 2/4\\
	50-384 Wrocław, Poland
}
\thanks{The first author is supported by the Narodowe Centrum Nauki grants 2015/19/B/ST1/01151 and 2016/22/E/ST1/00450.}

\author{Anand Pillay}
\email[A.\ Pillay]{apillay@nd.edu}
\address[A.\ Pillay]{Department of Mathematics, University of Notre Dame\\
	281 Hurley Hall\\
	Notre Dame, IN 46556, USA}
\thanks{The second author is supported by NSF grants DMS-1360702 and DMS-1665035.}



\date{}

\begin{document}
	
	\begin{abstract}
We prove several theorems relating {\em amenability}  of groups in  various categories (discrete, definable, topological, automorphism group) to  {\em model-theoretic invariants} (quotients by connected components, Lascar Galois group, G-compactness, ...).  Among the main tools, which is possibly of independent interest, is the adaptation and generalization of  Theorem 12 of \cite{MaWa} to other settings including amenable topological groups whose topology is generated by a family of subgroups. 

On the side of definable groups, we prove that if $G$ is definable in a model $M$ and $G$ is definably amenable, then the connected components ${G^{*}}^{00}_{M}$ and ${G^{*}}^{000}_{M}$ coincide, answering positively a question from \cite{KrPi}. We also prove some natural counterparts for topological groups, using our generalizations of \cite{MaWa}. By finding a dictionary relating quotients by connected components and Galois groups of $\omega$-categorical theories, we conclude that if $M$ is countable and $\omega$-categorical, and $\aut(M)$ is amenable as a topological group, then $T:=\Th(M)$ is {\em G-compact}, i.e. the Lascar Galois group $\gal_L(T)$ is compact, Hausdorff (equivalently, the natural epimorphism from $\gal_L(T)$ to $\gal_{KP}(T)$ is an isomorphism).  

We also take the opportunity to further develop the model-theoretic approach to topological dynamics,  obtaining for example some new invariants for topological groups, as well as allowing a uniform approach to the theorems above and the various categories.


	\end{abstract}

	\maketitle

	\section{Introduction}


Topological dynamics from the model-theoretic point of view has been mainly developed in the following three contexts: 
\begin{enumerate}
\item[(1)] for a group $G$ definable in a first order structure acting on some spaces of types (e.g. in \cite{Ne1,Ne2, KrPi, ChSi,Kr}), 
\item[(2)] for the group $\aut(\C)$ of automorphisms of a monster model $\C$ of a given theory  acting on a certain space of types \cite{KrPiRz}, 
\item[(3)] for groups of automorphisms of countable, $\omega$-categorical structures (e.g. in \cite{BIT}). 
\end{enumerate}
The main motivation is the fact that notions and ideas from topological dynamics lead to new interesting phenomena in model theory (a generalization of the theory of stable groups) and sometimes can be used as tools to solve open problems in model theory (e.g. in \cite{KrPiRz}, the topological dynamics of $\aut(\C)$ was used to prove very general theorems on the complexity of bounded, invariant equivalence relations). On the other hand,  one can hope to get some new insight into purely topological dynamical problems by an application of some knowledge from model theory (e.g. in \cite{BIT}, using non-trivial model theory, the authors found an example of an oligomorphic group whose various compactifications have some desirable properties, which had been an open problem).

Our general goal and motivation in this and in forthcoming papers is: in each of the following three contexts 
\begin{enumerate}
\item[(i)] $G$ is a group definable in a first order structure, 
\item[(ii)] $G= \aut(M)$, where $M$ is a countable, $\omega$-categorical structure,
\item[(iii)] $G=\aut(\C)$, where $\C$ is a monster model of a given theory,
\end{enumerate}
describe model-theoretic consequences of various dynamical properties of $G$, or even try to express such properties in  purely model-theoretic terms.
This can lead to new interactions and mutual applications between model theory and topological dynamics. Recall that an analogous approach initiated in \cite{KPT} concerning mutual translations of dynamical properties of the groups of automorphisms of Fraiss\'{e} structures and Ramsey-theoretic properties of the corresponding Fraiss\'{e} systems turned out to be very fruitful.  

Among our motivations for this paper was to find model-theoretic consequences of the assumption of amenability of $G$ in appropriate senses. We focus on contexts (i) and (ii). As to (i), we consider a significantly more general situation when $G$ is a topological group; some issues in this context have been already investigated in \cite{GPP}, \cite{PeSt} and \cite{Gi}; in other papers concerning the dynamics of $G$ in model theory, the topology on $G$ was not considered, or, in other words, $G$ was treated as a discrete group.

Amenable groups play a major role in mathematics, and our interest in this class does not require a justification. Recall that for a definable group $G$ there is also a more general notion of {\em definable amenability} (which is just the existence of a left $G$-invariant, finitely additive probability measure on the algebra of definable subsets of $G$); see Subsection \ref{subsection: variants amenability} for more notions of amenability used in this paper. On the model-theoretic side, we will focus on the notions of {\em G-compactness} and {\em G-triviality}, which we briefly discuss now (more details can be found in Subsection \ref{subsection: preliminaries on G-compactness}) .

Recall that with an arbitrary theory $T$ we can associate Galois groups $\gal_L(T)$ (the Lascar Galois group) and $\gal_{KP}(T)$ (the Kim-Pillay Galois group) which are invariants of $T$ (i.e. they do not depend on the choice of the monster model in which they are computed). There is a natural epimorphism from $\gal_L(T)$ to $\gal_{KP}(T)$ (whose kernel is the closure of the identity in the so-called logic topology) and it is interesting to understand when it is an isomorphism (which means that these two invariants of $T$ coincide), in which case we say that $T$ is {\em G-compact}; this is equivalent to saying that $\gal_L(T)$ is Hausdorff with the logic topology. For example, for the theory $T:=ACF_p$ of algebraically closed fields of characteristic $p$, $\gal_L(T)=\gal_{KP}(T)$ coincides with the absolute Galois group of the prime field, so $T$ is G-compact. 
G-compactness was introduced in  \cite{La}, where the Lascar Galois group of a complete first order theory also makes its first appearance.  In fact, in \cite{La}, a stronger definition was given, namely that after naming any finite set of parameters the natural map from $\gal_{L}$ to $\gal_{KP}$ is an isomorphism. (For some reason in \cite{LaPi} the original definition was weakened.)  It appears that the original motivation for introducing these notions in \cite{La} was connected with Michael Makkai's program or project of trying to recover an $\omega$-categorical theory $T$,  or rather its classifying topos $E(T)$, from the category $\Mod(T)$ of models of $T$, at least in special cases.   In \cite{La}, Lascar proves, modulo results of Makkai, that when $T$ is {\em G-finite} (and $T$ is $\omega$-categorical), then the program succeeds: one can recover $E(T)$ (so $T$) from $\Mod(T)$.  Here, G-finite means G-compact in the strong sense together with $\gal_{L}(T)$ being finite, even after adding parameters for finite sets. An account of the category-theoretic aspect of this result when $T$ is {\em G-trivial} (which means that $\gal_{L}(T)$ is trivial after naming any finitely many parameters) appears in \cite{Ma}.  In any case, via these results of Lascar and Makkai, Theorem \ref{theorem: new theorem} below will deduce from extreme amenability of the topological group $\aut(M)$ that $\Th(M)$ can be recovered from its category of models (when $M$ is countable, $\omega$-categorical).
Lascar strong types were also introduced in \cite{La}: they are classes of the finest bounded, ($\emptyset$-)invariant equivalence relation on a product of sorts of $\C$. In \cite{LaPi}, Kim-Pillay strong types were introduced: they are classes of the finest bounded, $\emptyset$-type-definable equivalence relation on a product of sorts of $\C$. These strong types (together with the well-known Shelah strong types) have played a fundamental role in model theory (particularly in stable, simple and NIP theories). It is well-known that G-compactness is equivalent to saying that Lascar strong types coincide with Kim-Pillay strong types. Since the space of Kim-Pillay strong types is  in general much nicer that the space of Lascar strong types, it is highly desirable to understand when both classes of types coincide. This motivates the interest in G-compactness.

In the context of a definable group $G$ (whose interpretation in the monster model is denoted by $G^*$), the counterparts of Galois groups are quotients of $G^*$ by appropriate connected components. Let $G$ be definable in $M$. By ${G^*}^{00}_M$ we denote the smallest $M$-type-definable subgroup of $G^*$ of bounded index, and by ${G^*}^{000}_M$ -- the smallest invariant over $M$ subgroup of bounded index. These are normal subgroups of $G^*$. The quotients $G^*/{G^*}^{00}_M$ and $G^*/{G^*}^{000}_M$ are invariants of $G$ (they do not depend on the choice of the monster model). In fact, they can be equipped with the logic topology, where a subset is closed if its preimage under the quotient map is type-definable. 
Then the former quotient becomes a compact topological group, whereas the latter one is a quasi-compact topological group (so not necessarily Hausdorff). 
It is Hausdorff if and only if ${G^*}^{000}_M = {G^*}^{00}_M$, or, in other words, if the natural epimorphism from $G^*/{G^*}^{000}_M$ to $G^*/{G^*}^{00}_M$ is an isomorphism, which corresponds to G-compactness.  It is worth emphasizing that with an arbitrary $G$ we associate a classical mathematical object $G^*/{G^*}^{00}_M$, namely a compact (Hausdorff) topological group, which is particularly nice in o-minimal structures due to the truth of Pillay's Conjecture; and this is not the case for $G^*/{G^*}^{000}_M$ when ${G^*}^{000}_M \ne {G^*}^{00}_M$.  Recall also from that after adding an affine copy of $G$ as a new sort, the orbits of ${G^*}^{000}_M$ and ${G^*}^{00}_M$ are exactly Lascar strong types and Kim-Pillay strong types over $M$ on this sort, respectively (see \cite[Lemma 3.7]{GiNe}). All of this shows that it is desirable to understand when ${G^*}^{000}_M = {G^*}^{00}_M$. 

Very roughly speaking, our main results are of the form: an appropriate amenability assumption implies an appropriate version of G-compactness.

Section \ref{section: dynamics of top groups}, having partly a preliminary character, contains many new notions and observations which are essential for the rest of the paper.
In Subsection \ref{subsection: topological category}, we develop the topological dynamics of a topological group $G$ treated as a first order structure where predicates for all open subsets of $G$ are in the language. As was noted in \cite{Gi,GPP}, there is a type-definable over $M$ (even over $\emptyset$) subgroup $H$ of $G^*$ such that the quotient map from $G$ to $G^*/H$ is the Bohr compactification of $G$. We denote this subgroup by ${G^*}^{00}_{\topo}$, and we give some descriptions of it; this is a topological variant of ${G^*}^{00}_M$. We also define a topological version of ${G^*}^{000}_M$, which we denote by ${G^*}^{000}_{\topo}$ and whose description as a universal object in a certain category we provide. 
We also recall from \cite{GPP, PeSt} a model-theoretic description of the universal $G$-ambit as a quotient of $G^*$ by a certain type-definable equivalence relation and we describe model-theoretically the semigroup operation on it. Then we notice that the main results from \cite{KrPi} have their counterparts for topological groups (which will not be used in the subsequent sections, but are interesting in their own right). We also give some basic examples illustrating the differences between the definable and topological categories. In particular, we compute various components for the universal cover of $\Sl_2(\mathbb{R})$, and obtain as a conclusion that the Bohr compactification of this cover is trivial, whereas the generalized Bohr compactification is non-trivial (see Corollary \ref{cor: classical application}).
In Subsection \ref{subsection: def top}, working with an arbitrary language (without predicates for all open subgroups of $G$), we introduce definable topological connected components ${G^*}^{00}_{\defi,\topo}$ and ${G^*}^{000}_{\defi,\topo}$, and we give their descriptions; and similarly for the universal definable topological $G$-ambit.

The main body of work is contained in Sections \ref{section: amenability of G} and \ref{section:amenability of Aut(M)}. Our first goal is Question 5.1 from \cite{KrPi} which we recall as a conjecture.

\begin{conj}\label{con: conjecture from KrPi}
Let $G$ be a group definable in a structure $M$. If $G$ is definably amenable, then ${G^*}^{00}_M={G^*}^{000}_M$.
\end{conj}
%
We should recall here that in \cite{KrPi} this conclusion was obtained for definably strongly amenable groups (in particular, for all nilpotent groups) in the sense of Glasner, assuming additionally that all types in $S_G(M)$ are definable (which is the case e.g. when predicates for all subsets of $G$ are in the language). On the other hand, in \cite{GiKr}, it was proved by completely different (algebraic) methods that when $G$ is a non-abelian free group or a surface group of genus at least 2, equipped with predicates for all subsets, then the two connected components in question are distinct. Question 4.28 from \cite{GiKr} asks for which groups equipped with predicates for all subsets are these two connected components the same. The result from \cite{KrPi} mentioned above implies that it is the case for all nilpotent groups, and Conjecture \ref{con: conjecture from KrPi} (which we will prove) implies that it is the case also for all solvable groups. Recall also that Conjecture \ref{con: conjecture from KrPi} was proved in \cite{HrPi} under the assumption that $\Th(M)$ has NIP, using the machinery of $f$-generic types which is not available for arbitrary structures. Hrushovski informed us that the case of Conjecture \ref{con: conjecture from KrPi} when predicates for all subsets of $G$ are in the language can be easily seen to follow from Theorem 3.5 of \cite{Hr}.

We can also formulate an obvious analogue of the above conjecture in the topological context.

\begin{conj}\label{con: topological version of the conjecture from KrPi}
Let $G$ be a topological group. If $G$ is amenable, then ${G^*}^{00}_{\topo}={G^*}^{000}_{\topo}$.
\end{conj}

To deal with this conjecture, we expand the group $G$ by predicates for all open subsets. This context extends the particular case of Conjecture  \ref{con: conjecture from KrPi} when predicates for all subsets of $G$ are in the language (namely, in such a situation, treating $G$ as a discrete group, we observe in Section \ref{section: dynamics of top groups} that ${G^*}^{00}_M={G^*}^{00}_{\topo}$ and ${G^*}^{000}_M={G^*}^{000}_{\topo}$).

As we will see in Subsection \ref{subsection: variants amenability},
a common generalization of Conjectures \ref{con: conjecture from KrPi} and \ref{con: topological version of the conjecture from KrPi} is the following (where weak definable topological amenability is defined in Definition \ref{definition: weak definable topological amenability} below).

\begin{conj}\label{con: the most general}
Let $G$ be a topological group definable in an arbitrary structure $M$. If $G$ is weakly definably topologically amenable, then ${G^*}^{00}_{\defi,\topo}={G^*}^{000}_{\defi,\topo}$.
\end{conj}

Even the following restriction of the above conjecture generalizes the previous conjectures. 

\begin{conj}\label{con: general enough}
Let $G$ be a topological group definable in a structure $M$ in which the members of a basis of open neighborhoods of $e$ are definable. If $G$ is definably topologically amenable, then ${G^*}^{00}_{\defi,\topo}={G^*}^{000}_{\defi,\topo}$.
\end{conj}

In Subsection \ref{subsection: amenability and components}, we prove Conjecture \ref{con: general enough} under the assumption that there is a basis of open neighborhoods of $e$ consisting of open subgroups which are definable. 

\begin{thm}\label{thm: main theorem in Section 2}
Let $G$ be a topological group definable in a structure $M$. Assume that $G$ has a basis of open neighborhoods of $e$ consisting of definable, open subgroups. 
If $G$ is definably topologically amenable, then ${G^*}^{00}_{\defi,\topo}={G^*}^{000}_{\defi,\topo}$.
\end{thm}

This implies Conjecture \ref{con: conjecture from KrPi} in its full generality and Conjecture \ref{con: topological version of the conjecture from KrPi} for groups possessing a basis of open neighborhoods of $e$ consisting of open subgroups (which is strong enough for an application in Section \ref{section:amenability of Aut(M)}, namely to prove Theorem \ref{thm: the main result of the paper}). 
The main tool in our proof is a technique from \cite{MaWa} for understanding ``approximate subgroups''. In the context of a definable, definably amenable group $G$ in an $\omega^+$-saturated structure, and a definable subset $X$ of $G$ with positive measure, they construct a chain of symmetric, generic, definable subsets $Y_{1} \supseteq Y_{2} \supseteq \dots$ of   $(XX^{-1})^{4}$ such that $Y_{i+1}^{2}\subseteq Y_{i}$ for all $i$. With additional properties this is sometimes called a (definable) Bourgain system.  
This is motivated by Theorem 3.5 of \cite{Hr}, and uses ideas of Sanders \cite{Sa}.  
We will adapt the technique to more general contexts, in particular to topological groups, dropping the $\omega^+$-saturation requirement.
We first prove Conjecture \ref{con: conjecture from KrPi} assuming that predicates for all subsets of $G$ are in the language, and then extend this argument to show Conjecture \ref{con: topological version of the conjecture from KrPi} under the assumption that there is a basis of open neighborhoods of $e$ consisting of open subgroups. We consider these special cases of Theorem \ref{thm: main theorem in Section 2}, because they can be obtained by a simplification of the argument from \cite{MaWa} and make the main ideas more transparent. Then we use the full power of the argument from \cite{MaWa} together with some new arguments to prove Theorem \ref{thm: main theorem in Section 2}.

In Subsection \ref{subsection: extreme amenability}, we give a quick proof of the the following counterpart of Conjectures \ref{con: general enough}, \ref{con: topological version of the conjecture from KrPi} and \ref{con: conjecture from KrPi} for extremely amenable groups, which is used to prove Corollary \ref{theorem: new theorem}.

\begin{prop}\label{prop: extremely amenable}
Let $G$ be a topological group definable in a structure $M$ such that there is a basis of open neighborhoods of $e$ consisting of definable sets. Then, if $G$ is definably topologically extremely amenable, then $G^*={G^*}^{000}_{\defi, \topo}= {G^*}^{00}_{\defi, \topo}$. In particular: if $G$ is extremely amenable as a topological group, then $G^*={G^*}^{000}_{\topo}= {G^*}^{00}_{\topo}$; if $G$ is definably extremely amenable, then $G^*={G^*}^{000}_{M}= {G^*}^{00}_{M}$.
\end{prop}

Section \ref{section:amenability of Aut(M)} contains the main result of this paper.

\begin{thm}\label{thm: the main result of the paper}
Let $M$ be a countable, $\omega$-categorical structure. If $\aut(M)$ is amenable (as a topological group), then the theory of $M$ is G-compact.
\end{thm}

Let $T$ be the theory of $M$.
In order to prove this theorem, we treat $G:=\aut(M)$ as a group definable in the structure 
${\mathcal M}$ consisting of the structure $M$ together with the group $G$ acting on $M$, expanded by predicates for all 
open subsets of $G$.
Then we find group isomorphisms $\rho \colon \gal_L(T) \to G^*/{G^*}^{000}_{\topo}$ and $\theta \colon \gal_{KP}(T) \to  G^*/{G^*}^{00}_{\topo}$ such that the following diagram commutes

\begin{figure}[H]
		\centering
		\begin{tikzcd}
			\gal_L(T) \arrow[r,"h"]\arrow[d,"\rho"]&\gal_{KP}(T)\arrow[d,"\theta"] \\
			G^*/{G^*}^{000}_{\topo}\arrow[r] & G^*/{G^*}^{00}_{\topo},
		\end{tikzcd}
	\end{figure}
\noindent
where the horizontal maps are the obvious epimorphisms.
Note that $G=\aut(M)$ has a basis of open neighborhoods of the identity consisting of open subgroups.
Thus, using the above diagram together with Theorem \ref{thm: main theorem in Section 2} (more precisely, we use Conjecture \ref{con: topological version of the conjecture from KrPi} for groups possessing a basis of open neighborhoods of the identity consisting of open subgroups), we get that $h$ is an isomorphism, i.e. $T$ is G-compact. We find this method interesting in its own right, as it provides a dictionary between Galois groups and quotients by connected components in the $\omega$-categorical world. Udi Hrushovski has suggested to us a proof of Theorem \ref{thm: the main result of the paper} using alternative methods.

At the very end of this paper, using Proposition \ref{prop: extremely amenable} and the existence of the isomorphism $\rho$, we obtain the following corollary, although it can be easily seen directly, as Hrushovski mentioned to us, and is even implicit in \cite{HrPi}.

\begin{cor}\label{theorem: new theorem}
Let $M$ be a countable, $\omega$-categorical structure. If $\aut(M)$ is extremely amenable (as a topological group), then the theory of $M$ is G-trivial.
\end{cor}

Theorem \ref{thm: the main result of the paper} and Corollary \ref{theorem: new theorem} together with the Kechris, Pestov, Todor\v cević machinery \cite{KPT} (proving that groups of automorphisms of various  Fraiss\'{e} structures are [extremely] amenable) can be used to show that some $\omega$-categorical structures are G-compact or even G-trivial. For example, by \cite[Theorem 6.14]{KPT}, we know that the automorphism group of the {\em countable atomless Boolean algebra with the canonical ordering} is extremely amenable, so the theory of this algebra is G-trivial by Corollary \ref{theorem: new theorem}.


\section{A few preliminaries}\label{section: preliminaries}

For a detailed exposition of preliminaries concerning the topological dynamics of definable groups the reader is referred to Section 1 in \cite{KrPi}. Here, we only recall a few basic things. Some other definitions are recalled in the appropriate places of this paper.

In the whole paper, for a group [or a set] $G$ definable in a structure $M$, by $G^*$ we denote the interpretation of $G$ in the monster model in which we are working (i.e. a $\kappa$-saturated and strongly $\kappa$-homogeneous elementary extension of $M$ for a ``sufficiently large'' strong limit cardinal $\kappa$). The monster model will usually be denoted by $\C$.  Partial types (over parameters) will often be identified with sets of their realizations in $\C$ which are called type-definable sets.
An invariant (i.e. invariant under $\aut(\C)$) equivalence relation on a  type-definable subset of a product of a small  number $\lambda$ (i.e. $\lambda <\kappa$) of sorts of $\C$ is said to be {\em bounded} if it has less than $\kappa$ many classes (equivalently, at most $2^{|T|+\lambda}$ classes).
	
	\begin{dfn}
		Let $E$ be a bounded, invariant equivalence relation on a type-definable subset $P$ of some product of sorts of $\C$. We define the {\em logic topology} on $P/E$ by saying that a subset $D \subseteq P/E$ is closed if its preimage in $P$ is type-definable.
	\end{dfn}

We follow the convention that compact spaces are Hausdorff by definition; ``compact'' spaces which are not necessarily Hausdorff will be called {\em quasi-compact}. 	
	It is folklore that $P/E$ is quasi-compact, and $P/E$ is Hausdorff if and only if $E$ is type-definable. In the case when $P$ is a definable group and $H$ is an invariant, bounded index, normal subgroup, the quotient $P/H$ is a quasi-compact (so not necessarily Hausdorff) topological group; it is Hausdorff if and only if $H$ is type-definable. If $\C$ is a monster model in languages ${\mathcal L} \subseteq {\mathcal L'}$ and $E$ is bounded and type-definable in ${\mathcal L}$, then the logic topologies on $P/E$ computed in ${\mathcal L}$ and ${\mathcal L'}$ coincide, because they are compact and the latter one is stronger than the former. In fact, this is true even if $E$ is only invariant in ${\mathcal L}$ (and bounded), because if $D \subseteq P/E$ is closed in the logic topology computed in ${\mathcal L'}$, then the preimage of $D$, being type-definable in ${\mathcal L'}$ and invariant in the sense of ${\mathcal L}$ over any given model $M$ (which follows from the assumption that $E$ is bounded and invariant in ${\mathcal L}$), must be $M$-type-definable in ${\mathcal L}$, so $D$ is closed in the logic topology computed in ${\mathcal L}$.

Recall from \cite{GPP} or \cite{KrPi} that for a group $G$ definable in a structure $M$, a map $f \colon G \to C$, where $C$ is a compact (Hausdorff) space, is said to be {\em definable} if for any disjoint, closed subsets $C_1$ and $C_2$ of $C$, the preimages $f^{-1}[C_1]$ and $f^{-1}[C_2]$ can be separated by a definable set. By \cite[Lemma 3.2]{GPP}, we know that this happens if and only if $f$ is the restriction of a map $f^* \colon G^* \to C$ which is {\em $M$-definable} in the sense that the preimage under $f^*$ of any closed subset of $C$ is type-definable over $M$. If such a map $f^*$ exists, it is unique and it is given by the formula $$f^*(a)=\bigcap_{\varphi \in \tp(a/M)} \cl(f[\varphi(M)]).$$

Let $G$ be a topological group. Recall that a (topological) {\em $G$-flow} is a pair $(G,X)$, where $X$ is a compact (Hausdorff) space on which $G$ acts continuously. A {\em $G$-ambit} is a $G$-flow with a distinguished point whose $G$-orbit is dense. Suppose $G$ is a definable group (in some structure). In the so-called ``definable category'', the topology on $G$  is irrelevant (i.e. $G$ is treated as a discrete group), and a $G$-flow $(G,X)$ is said to be {\em definable} if for any $x \in X$ the map $f_x\colon G \to X$ given by $f_x(g)=gx$ is definable.

\section{Topological dynamics for topological groups via model theory}\label{section: dynamics of top groups}

The goal of this section is to recall and extend the model-theoretic approach to the topological dynamics of a topological group $G$ from \cite{GPP} and \cite{PeSt}. Facts \ref{fac: from GPP}, \ref{fct: continuity of the ambit} and \ref{fac: universal G-ambit} come from \cite{GPP} and \cite{PeSt}; Proposition \ref{prop: characterization of G^00_top}, Corollary \ref{cor: formula for G^00_top} and Proposition \ref{prop: description of *} are new. In the rest of this section, we introduce new notions (mainly connected components in various categories) and we obtain some results about them. (As an application, in Corollary \ref{cor: classical application}, we get new information about certain classical compactifications of the universal cover of $\Sl_2(\R)$.) The material developed in this section is essential to formulate the main results of Section \ref{section: amenability of G} and to prove the main results of this paper in Section \ref{section:amenability of Aut(M)}.

\subsection{The topological category}\label{subsection: topological category}

Throughout this subsection, let $G$ be a topological (so Hausdorff) group which is $\emptyset$-definable in a first order structure $M$, and assume that predicates for all open subsets of $G$ are in the language.	

\begin{rem}\label{rem: continuous implies definable}
Each continuous function $f \colon G \to C$ from $G$ to a compact space $C$ is definable. In particular, each (topological) $G$-flow is definable.
\end{rem} 

\begin{proof}
This follows from the fact that the preimage of every closed set is closed and so definable.
\end{proof}

Because of this remark, in this subsection we just work in the category of (topological) $G$-flows.
The next lemma is an improvement of \cite[Lemma 3.2(i)]{GPP} for topological groups.

\begin{lem}\label{lem: extending f}
Suppose $f \colon G \to C$ is a continuous map from $G$ to a compact space $C$. Then $f$ extends uniquely to an $M$-definable map $f^* \colon G^* \to C$. Moreover, $f^*$ is given by the formula $f^*(a)=\bigcap\{\cl(f[U]): U \in \tp(a/M)\;\, \mbox{is open}\}$. Furthermore, $f^*$ is $M$-definable in a strong sense, namely, for any closed $F \subseteq C$ the preimage ${f^*}^{-1}[F]$ is the intersection of sets of the form $U^*$ for some open subsets $U$ of $G$; more precisely, ${f^*}^{-1}[F]= \bigcap \{ f^{-1}[V]^*: V \supseteq F \;\, \mbox{is open}\}$. 
\end{lem}

\begin{proof}
The existence and uniqueness of $f^*$ follows from \cite[Lemma 3.2(i)]{GPP} and Remark \ref{rem: continuous implies definable}. By the same lemma, we also know that $f^*(a)=\bigcap\{\cl(f[\varphi(M)]): \varphi(x) \in \tp(a/M)\}$. (By the way, in the comment between parenthesis in \cite{GPP} on the proof that $f^*$ is definable over $M$, $\Sigma(y)$ should be the collection of formulas $\varphi(y)$ over $M$ such that $f^{-1}[D'] \subseteq \varphi(M)$ for all closed $D'\supseteq D$ for which there is an open set $U$ such that $D \subseteq U \subseteq D'$.) 

So, for the ``moreover'' part, it is enough to show that the set $\bigcap\{\cl(f[U]): U \in \tp(a/M)\;\, \mbox{is open}\}$ is a singleton. Assume for a contradiction that there are distinct elements $x,y \in \bigcap\{\cl(f[U]): U \in \tp(a/M)\;\, \mbox{is open}\}$. One can find open sets $U_1,V_1,U_2,V_2$ and closed sets $C_1,C_2$ such that $x \in U_1\subseteq C_1\subseteq V_1$, $y \in U_2\subseteq C_2\subseteq V_2$ and $V_1 \cap V_2 = \emptyset$. Then the sets $f^{-1}[V_1]$ and $f^{-1}[V_2]$ are open (so definable) and disjoint. Thus, one of them does not belong to $\tp(a/M)$. Without loss $f^{-1}[V_1] \notin \tp(a/M)$. Then $f^{-1}[V_1]^c \in \tp(a/M)$. Put $V_1'=f^{-1}[C_1]^c$. Then $V_1'$ is open and $f^{-1}[V_1]^c \subseteq V_1'$, so $V_1' \in \tp(a/M)$, and hence $x \in \cl (f[V_1']) = \cl (f[f^{-1}[C_1]^c]) \subseteq \cl (C_1^c) \subseteq \cl (U_1^c) = U_1^c$, a contradiction.

In remains to prove the ``furthermore'' part, i.e. the equality $${f^*}^{-1}[F]= \bigcap \{ f^{-1}[V]^*: V \supseteq F \;\, \mbox{is open}\}.$$

$(\subseteq)$ Suppose for a contradiction that for some $a \in G^*$ we have $f^*(a) \in F$ but $a \notin \bigcap \{ f^{-1}[V]^*: V \supseteq F \;\, \mbox{is open}\}$. Then $a \in {f^{-1}[V]^*}^c$ for some open $V \supseteq F$. Choose an open set $U_1$ and a closed set $C_1$ such that $F \subseteq U_1 \subseteq C_1 \subseteq V$. Then $f^{-1}[V]^c \subseteq f^{-1}[C_1]^c$ and the last set is open, so $a \in  (f^{-1}[C_1]^c)^*$, and hence $f^*(a) \in \cl (f[ f^{-1}[C_1]^c]) \subseteq \cl (C_1^c) \subseteq \cl (U_1^c)=U_1^c$. Therefore $f^*(a) \notin F$, a contradiction.

($\supseteq$) Suppose for a contradiction that $a \in \bigcap \{ f^{-1}[V]^*: V \supseteq F \;\, \mbox{is open}\}$, but $f^*(a) \notin F$. Take open and disjoint sets $V_1$ and $V_2$ such that $F \subseteq V_1$ and $f^*(a) \in V_2$. Then $a \in f^{-1}[V_1]^*$, so $f^*(a) \in \cl(f[f^{-1}[V_1]]) \subseteq \cl (V_1) \subseteq V_2^c$, a contradiction. 
\end{proof}


Recall that for a set of parameters $A$, ${G^*}^{00}_A$ denotes the smallest $A$-type-definable, bounded index subgroup of $G^*$.

\begin{dfn}
We define ${G^*}^{00}_{\topo}$ to be the smallest bounded index subgroup of $G^*$ which is an intersection of some sets of the form $U^*$ for $U$ open in $G$. Let $\mu$ denote the intersection of all $U^*$'s for $U$ ranging over all open neighborhoods of $e$.
\end{dfn}

We see that ${G^*}^{00}_{\topo}$ is type-definable over $\emptyset$, and so ${G^*}^{00}_{\topo}\geq {G^*}^{00}_\emptyset$. $\mu$ is also a subgroup which is type-definable over $\emptyset$, but it may be of unbounded index. Clearly, $\mu \leq {G^*}^{00}_{\topo}$. It is also easy to see directly from the definition that both $\mu$ and ${G^*}^{00}_{\topo}$ are normalized by $G$.

Recall that a {\em group compactification} of $G$ is a continuous homomorphism from $G$ to a compact (Hausdorff) group $K$ with dense image (or just this compact group $K$). (For convenience we will write ``compactification'' instead of ``group compactification''.) There is always a unique up to isomorphism universal compactification of $G$, and it is called the {\em (topological) Bohr compactification} of $G$.

The next fact is Proposition 2.1 from \cite{GPP} which gave \cite{Gi} as a reference for the proof. We give a direct proof for the reader's convenience (as we refer to this proof later several times) and to show that it can be obtained by the methods from Section 3 of \cite{GPP}.

\begin{fct}\label{fac: from GPP}
i) ${G^*}^{00}_{\topo}$ is a normal subgroup of $G^*$.\\
ii) The quotient mapping $\pi \colon G \to G^*/{G^*}^{00}_{\topo}$ is the Bohr compactification of $G$.
\end{fct}

\begin{proof}
(i) We have that ${G^*}^{00}_{\topo} =\bigcap\{ U^*: U \in {\mathcal U}\}$, where $\mathcal U$ is a family of some open neighborhoods of $e$ in $G$. We can assume that $\mathcal U$ is maximal with this property. Let $R$ be the set of representatives of right cosets of ${G^*}^{00}_{\topo}$ in $G^*$. Then 
$$H:=\bigcap_{g \in G^*}({G^*}^{00}_{\topo})^g=\bigcap_{g \in R}({G^*}^{00}_{\topo})^g,$$
so $H$ is invariant (over $\emptyset$) and type-definable, and hence it is $\emptyset$-type-definable.

Take any $\varphi(x) \in H(x)$. Then, by compactness, there are $U_1,\dots,U_n \in \mathcal U$ and $a_1,\dots,a_n \in R$ such that ${U_1^*}^{a_1} \cap \dots \cap {U_n^*}^{a_n} \subseteq \varphi(G^*)$.  So there are $g_1,\dots,g_n \in G$ for which $U_1^{g_1}\cap \dots \cap U_n^{g_n} \subseteq  \varphi(G)$; then $(U_1^{g_1})^* \cap \dots \cap (U_n^{g_n})^* \subseteq \varphi(G^*)$. On the other hand, since ${G^*}^{00}_{\topo}$ is normalized by $G$, we see that $\mathcal U$ is closed under conjugation by the elements of $G$. Therefore, ${G^*}^{00}_{\topo} \subseteq \varphi(G^*)$. Thus, we have proved that ${G^*}^{00}_{\topo} = H$, which means that ${G^*}^{00}_{\topo}$ is normal.\\[1mm]
(ii) The density of $\pi[G]$ in $G^*/{G^*}^{00}_{\topo}$ equipped with the logic topology is folklore (the preimage under the quotient map of a non-empty open set is a $\bigvee$-definable over $M$ subset of $G^*$, so it has a point in $G$). 

To see that $\pi$ is continuous, note that if $g \in \pi^{-1}[V]$ for an open $V \subseteq G^*/{G^*}^{00}_{\topo}$, then, by the definition of ${G^*}^{00}_{\topo}$ and compactness, there is an open set $U \subseteq G$ such that $gU^*/{G^*}^{00}_{\topo} \subseteq V$, so $gU \subseteq \pi^{-1}[V]$.

To get that $\pi$ is universal, we will apply the proof of \cite[Proposition 3.4]{GPP}. Consider any compactification $f\colon G \to C$. By Lemma \ref{lem: extending f}, there is a unique $M$-definable (even $\emptyset$-definable) $f^*\colon G^* \to C$ extending $f$. 
By the proof of \cite[Proposition 3.4]{GPP}, we know that $f^*$ is a group homomorphism. By the last part of Lemma \ref{lem: extending f}, we conclude that $\ker (f^*)$ is a  bounded index, normal subgroup which is an intersection of some sets of the form $U^*$ for $U$ open in $G$. Since ${G^*}^{00}_{\topo}$ is the smallest such a group, we finish as in the proof of \cite[Proposition 3.4]{GPP}. Namely, there is a natural continuous homomorphism from $G^*/{G^*}^{00}_{\topo}$ to $G^*/\ker(f^*)$, and $G^*/\ker(f^*)$ is naturally topologically isomorphic with $C$, so we get a continuous homomorphism  from $G^*/{G^*}^{00}_{\topo}$ to $C$ which commutes with $\pi$ and $f$.
\end{proof}

Now, we give an equivalent description of ${G^*}^{00}_{\topo}$.

\begin{prop}\label{prop: characterization of G^00_top}
${G^*}^{00}_{\topo}$ is the smallest $M$-type-definable, bounded index subgroup of $G^*$ which contains $\mu$.
\end{prop}

\begin{proof}
Denote this smallest subgroup by ${G^*}^{00-}_{\topo}$. Clearly, ${G^*}^{00-}_{\topo}\leq {G^*}^{00}_{\topo}$. To show the opposite inclusion, we need to check that  ${G^*}^{00-}_{\topo}$ is a normal subgroup of $G^*$ and that the natural map $\pi^-:G \to  G^*/{G^*}^{00-}_{\topo}$ is a compactification of $G^*$. Indeed, then, by Fact \ref{fac: from GPP}, we have a continuous homomorphism $\sigma :  G^*/{G^*}^{00}_{\topo} \to G^*/{G^*}^{00-}_{\topo}$ commuting with the natural maps from $G$. But we also have a natural continuous homomorphism $\tau: G^*/{G^*}^{00-}_{\topo} \to G^*/{G^*}^{00}_{\topo}$ which clearly commutes with the maps from $G$. Thus, $\sigma \circ \tau$ is the identity on $\pi^-[G]$, so it is the identity on $G^*/{G^*}^{00-}_{\topo}$. Hence, $\tau$ is injective, and so ${G^*}^{00-}_{\topo}= {G^*}^{00}_{\topo}$. 

To see that ${G^*}^{00-}_{\topo}$ is a normal subgroup of $G^*$, one should apply a similar argument to the above proof that ${G^*}^{00}_{\topo}$ is normal. 
For this, first notice that since $\mu$ is normalized by $G$, so is ${G^*}^{00-}_{\topo}$. Next, write ${G^*}^{00-}_{\topo} = \bigcap\{ U^*: U \in {\mathcal U}\}$ for some family $\mathcal U$ of definable in $M$ subsets of $G$, and choose $\mathcal{U}$ maximal with this property. Then $\mathcal{U}$ is closed under conjugation by the elements of $G$, and so the proof of normality of ${G^*}^{00}_{\topo}$ goes through.

The fact that $\pi^-$ is a compactification of $G$ means that $\pi^-[G]$ is dense and that $\pi^-$ is continuous. The first part is folklore (the same arguments as for $\pi$). Continuity of $\pi^-$ also follows as for $\pi$ (using the assumption that $\mu \leq {G^*}^{00-}_{\topo}$).
\end{proof}

\begin{cor}\label{cor: formula for G^00_top} Recall that in this subsection we always assume that predicates for all open subsets of $G$ are in the language.\\
i) ${G^*}^{00}_{\topo}$ and $G^*/{G^*}^{00}_{\topo}$ (as a topological group with the logic topology) do not depend on the choice of the language.\\
ii) ${G^*}^{00}_{\topo}=\mu \cdot {G^*}^{00}_{\emptyset}=\mu \cdot {G^*}^{00}_{M}$.
\end{cor}

\begin{proof}
(i) is clear from the definition of ${G^*}^{00}_{\topo}$.\\
(ii) follows from Proposition \ref{prop: characterization of G^00_top} and normality of ${G^*}^{00}_{\emptyset}$ and ${G^*}^{00}_{M}$.  
\end{proof}

Working in the category of definable maps, we have the obvious notion of the {\em definable Bohr compactification} of $G$. In \cite{GPP}, it was proved that the quotient map from $G$ to $G^*/{G^*}^{00}_M$ is the definable Bohr compactification of $G$. Now, we give an example of a topological group $G$ for which the  Bohr compactification differs from the definable Bohr compactification.

\begin{ex}\label{ex: Anand 1}
It is well-known (e.g. see \cite[Chapter VII, Section 5]{Ka}) that the Bohr compactification of a locally compact, abelian group $A$ is the Pontryagin dual of $A^*$ treated as a discrete group, where $A^*$ is the Pontryagin dual of $A$ treated as a topological group. Consider $A=(\mathbb R,+)$. 

For $A$ treated as a topological group, the Pontryagin dual $A^*$ is isomorphic to $A$, and the Pontryagin dual of $A^*$ treated as a discrete group is the group $B$ of all (not necessarily continuous in the usual topology on $A^*$) homomorphisms from $A$ to the circle group $S^{1}$, which has cardinality $2^{\mathfrak{c}}$.

For $A$ treated as discrete group, the Pontryagin dual is the above group $B$, and the Pontryagin dual of $B$ treated as a discrete group has cardinality $2^{2^{\mathfrak{c}}}$.

Thus, the Bohr compactification of $A$ as a topological group differs from the Bohr compactification of $A$ as a discrete group. In particular, for $G:=A$ and $M$ being $G$ expanded by predicates for all subsets of $G$, the latter compactification is the definable Bohr compactification of $G$, and we conclude that ${G^*}^{00}_M \subsetneq {G^*}^{00}_{\topo}$.
\end{ex}

Now, we repeat and elaborate slightly on the description of the universal $G$-ambit from \cite{GPP} and \cite{PeSt}.

Define $E_\mu$ to be the finest bounded, $M$-type-definable equivalence relation on $G^*$ containing the equivalence relation $\sim$ defined by 
$$a \sim b \iff ab^{-1} \in \mu.$$

Define $S_G^\mu(M)$ to be the quotient $S_G(M)/\!\!\sim_\mu$, where $\sim_\mu$ is the equivalence relation on $S_G(M)$ defined by
$$p \sim_\mu q \iff \mu \cdot p=\mu \cdot q$$
($\mu \cdot p$ denotes the type-definable set $\mu \cdot p(G^*)$, or, equivalently, the partial type defining this set).
Let $h \colon G^* \to S_G^\mu(M)$ be given by $h(a)=\mu \cdot \tp(a/M)$. We see that $h(a)=h(b)$ if and only if there is $b' \in G^*$ such that $b' \equiv_M b$ and $a \in \mu \cdot b'$; hence, $h(a)=h(b)$ is a type-definable equivalence relation, which implies that $\sim_\mu$ is a closed relation.
Thus, $S_G^\mu(M)$ equipped with the quotient topology is a compact space. $G^*/E_\mu$ considered with the logic topology is also compact. We see that the equivalence relation of lying in the same fiber of $h$ is exactly the composition $\sim \circ \equiv_M$. So, by the definition of $E_\mu$, we get

\begin{rem}\label{rem: E_mu as composition} 
$E_\mu=\;\sim \circ \equiv_M$.
\end{rem}

Note that for any $r \in S_G(M)$ the equivalence class $r/\!\! \sim_{\mu} \in S_G^\mu(M)$ consists of all complete types over $M$ extending the partial type $\mu \cdot r$; so we will freely identify $r/\!\!\sim_{\mu}$ with $\mu \cdot r$.

Now, it is easy to see that $G$ acts on both $G^*/E_\mu$ and $S_G^\mu(M)$, respectively  by 
$$g (a/E_\mu) = (ga)/E_\mu \;\; \mbox{and}\;\; g (\mu \cdot p)=\mu \cdot (gp).$$

The next remark follows immediately from Remark \ref{rem: E_mu as composition}, compactness of the spaces in question, and the above definitions of the actions of $G$.

\begin{rem}
The function $f \colon G^*/E_\mu \to S_G^\mu(M)$ given by $f(a/E_\mu)=\mu \cdot \tp(a/M)$ is a well-defined homeomorphism preserving the actions of $G$.
\end{rem}

The next fact is Claim A.5 from \cite{PeSt}.

\begin{fct}\label{fct: continuity of the ambit}
The action of $G$ on $S_G^\mu(M)$ is continuous.
\end{fct}

Thus, the action of $G$ on $G^*/E_\mu$ is also continuous, and $(G, S_G^\mu(M),\mu \cdot \tp(e/M))$ and $(G,G^*/E_\mu,e/E_\mu)$ are isomorphic $G$-ambits, which will be identified from now on.



The next fact is Proposition 2.2 from \cite{GPP}. We give a proof to show that, as for ${G^*}^{00}_{\topo}$ above, the universality property can be also proved by the methods from Section 3 of \cite{GPP} and because we will refer to this proof in the proof of Proposition \ref{prop: description of the universal definable topological ambit}. This is also done in a similar way in \cite{PeSt}. 

\begin{fct}\label{fac: universal G-ambit}
$(G,G^*/E_\mu,e/E_\mu)$ is the universal (topological) $G$-ambit.
\end{fct}

\begin{proof}
Let $(G,X,x_0)$ be a (topological) $G$-ambit. Then $f \colon G \to X$ given by $f(g)=gx_0$ is continuous (and so definable by Remark \ref{rem: continuous implies definable}). By Lemma \ref{lem: extending f}, $f$ extends uniquely to an $M$-definable function $f^* \colon G^* \to X$. 
We want to show that $f^*$ factors through $E_\mu$. For this, by Remark \ref{rem: E_mu as composition}, it is enough to show that each of the conditions  $a\sim b$ and $a \equiv_M b$ implies $f^*(a)=f^*(b)$. That $a \equiv_M b$ implies $f^*(a)=f^*(b)$ follows from the fact that the fibers of $f^*$ are invariant over $M$. Now, assume $a \sim b$. Suppose for a contradiction that $f^*(a) \ne f^*(b)$. By the explicit formula for $f^*$ and compactness of $X$, we get formulas $\varphi \in \tp(a/M)$ and $\psi \in \tp(b/M)$ such that $\cl (f[\varphi(M)]) \cap \cl (f[\psi(M)]) = \emptyset$. By compactness of $X$ and continuity of the action of $G$ on $X$, we conclude that there is an open neighborhood $U$ of $e$ such that $U\varphi(M)x_0 \cap \psi(M)x_0 =\emptyset$. Hence, $U\varphi(M) \cap \psi(M)=\emptyset$, which implies that $\psi(M)\varphi(M)^{-1} \cap U = \emptyset$, but this contradicts the fact that $ab^{-1} \in \psi(M^*)\varphi(M^*)^{-1} \cap \mu$.

We have shown that $f^*$ factors through $E_\mu$, inducing a continuous map $\bar f^* : G^*/E_\mu \to X$ which commutes with the maps from $G$.
Finally, for any $g \in G$, $a \in G^*$, and $p:=\tp(a/M)$, the computation in the last sentence of the proof of \cite[Proposition 3.8]{GPP} yields $gf^*(a)=g\bigcap_{\varphi(x) \in p} \cl(f[\varphi(M)]) = \bigcap_{\varphi(x) \in p} g\cl(f[\varphi(M)])= \bigcap_{\varphi(x) \in p} \cl(f[g\varphi(M)]) = \bigcap_{\psi(x) \in gp} \cl(f[\psi(M)])  =f^*(ga)$, which immediately implies that $g\bar f^*(a/E_\mu)=\bar f^*(ga/E_\mu)$. 
\end{proof}


The following corollary follows easily from Fact \ref{fac: universal G-ambit}. 

\begin{cor}\label{cor: independence of the language}
The relation $E_\mu$ (and so $G^*/E_\mu$ as well) does not depend on the choice of the language (assuming that predicates for all open subsets of $G$ are in the language).
\end{cor}

Now, we give an example where the universal definable $G$-ambit is ``strictly bigger'' than the universal (topological) $G$-ambit.

\begin{ex}
Let $G$ be an infinite compact (Hausdorff) topological group. Then the universal $G$-ambit is just $(G,G,e)$. Consider $G$ as a group definable in the structure $M$ whose universe is $G$ equipped with the group operation and predicates for all subsets of $G$. Then, by \cite{GPP}, the universal definable $G$-ambit is the space $(G,S_G(M),\tp(e/M))$ which coincides with  $(G,\beta G,e)$ (where $\beta G$ is the Stone-$\check{\mbox{C}}$ech compactification of $G$ treated as a discrete group). Clearly, $\beta G$ is ``strictly bigger'' than $G$, as it has more elements. 
\end{ex}

By Fact \ref{fac: universal G-ambit}, we get a semigroup operation $*$ on $G^*/E_\mu$ given be
$$a/E_\mu * b/E_\mu = \lim_{g \to a/E_\mu} gb/E_\mu,$$
where the $g$'s in the limit are from $G$. Recall that $*$ is continuous on the left. We also get a natural action of $(G^*/E_\mu,*)$ on any $G$-ambit, and this action is also continuous on the left.

In order to use $*$ in model theory, we need to have  a description of $*$ in terms of realizations of types (as in the discrete case). This is done in the next proposition. 

\begin{prop}[Description of $*$ in $S_G^\mu(M)$]\label{prop: description of *}
For $p,q \in S_G(M)$, 
$$(\mu \cdot p) * (\mu \cdot q) = \mu \cdot \tp(ab/M),$$ 
where $a \models p$, $b \models q$ and $\tp(a/M,b)$ is finitely satisfiable in $M$. 
\end{prop}

\begin{proof}
Consider any basic open neighborhood of $\mu \cdot \tp(ab/M)$, i.e. a set of the form $U_\varphi^\mu:=\{\mu \cdot r \in S_G^\mu(M): \mu \cdot r \vdash \varphi\}$ for some formula $\varphi$ with parameters from $M$ such that $\mu \cdot \tp(ab/M) \in U_\varphi^\mu$. Then $\mu \cdot \tp(ab/M) \vdash \varphi$, so there are formulas $\theta \in \mu$ and $\psi \in \tp(ab/M)$ implying the formula defining $G$ and such that $(\theta \cdot \psi)(x) \vdash \varphi(x)$. Therefore,
$$\models (\theta(v) \wedge \psi(wb)) \rightarrow \varphi(vwb).$$

Consider any formula $\delta(w) \in p$ implying the formula defining $G$. Since $\delta(w) \wedge \psi(wb) \in \tp(a/bM)$, and the last type is finitely satisfiable in $M$, there is $g_{\varphi,\delta} \in G$ for which 
$$\models \delta(g_{\varphi,\delta}) \wedge \psi(g_{\varphi,\delta}b).$$

We conclude that $\models \theta(v) \rightarrow \varphi(vg_{\varphi,\delta}b)$, so 
$$g_{\varphi,\delta} (\mu \cdot q)=g_{\varphi,\delta}(\mu \cdot \tp(b/M))=\mu \cdot \tp(g_{\varphi,\delta}b/M) \in U_\varphi^\mu.$$
Hence, $$\lim_{\varphi,\delta} g_{\varphi,\delta} (\mu \cdot q)=\mu \cdot \tp(ab/M),$$ where the limit is taken with respect to the obvious directed set consisting of pairs of formulas $\varphi$ and $\delta$ as above (i.e. $\varphi$'s are such that $\mu \cdot \tp(ab/M) \in U_\varphi^\mu$ and $\delta \in p$). On the other hand, since $\models \delta(g_{\varphi,\delta})$, we have that 
$$\lim_{\varphi,\delta} \mu \cdot \tp(g_{\varphi,\delta}/M) = \mu \cdot p,$$
so $\lim_{\varphi,\delta} g_{\varphi,\delta} (\mu \cdot q)=(\mu \cdot p)*(\mu \cdot q)$. Therefore, $(\mu \cdot p) * (\mu \cdot q) = \mu \cdot \tp(ab/M)$.
\end{proof}

Recall that for  $A\subseteq \C$, ${G^*}^{000}_A$ denotes the smallest $A$-invariant, bounded index subgroup of $G^*$.
Now, we define a topological variant of this component.

\begin{dfn}
We define ${G^*}^{000}_{\topo}$ to be the smallest normal, bounded index, invariant over $M$ subgroup of $G^*$ which contains $\mu$.
\end{dfn}


This definition is in the spirit of Proposition \ref{prop: characterization of G^00_top}. However, there is a delicate issue here. Namely, it is not clear to us whether we can drop the normality assumption in this definition.

\begin{ques}
Is the smallest bounded index, invariant over $M$ subgroup of $G^*$ which contains $\mu$ normal in $G^*$? 
\end{ques}


If we knew that the word ``normal'' can be removed from the definition, we would immediately get ${G^*}^{000}_{\topo}=\mu \cdot {G^*}^{000}_M$. From the current definition,  we get the following, a bit more complicated description, which however is good enough for further applications.

\begin{rem}\label{rem: description of G^000}
${G^*}^{000}_{\topo}=\langle \mu^{G^*} \rangle \cdot {G^*}^{000}_M$, where $\langle \mu^{G^*} \rangle$ denotes the normal closure of $\mu$.
\end{rem}

Let $\hat{f} \colon S_G^\mu(M) \to G^*/{G^*}^{000}_{\topo}$ be defined by 
$$\hat{f}(\mu \cdot \tp(a/M))=a/{G^*}^{000}_{\topo}.$$ 
By the last remark and the fact that $a\equiv_M b$ implies $a^{-1}b \in {G^*}^{000}_{M}$, this is a well-defined function. (We used the symbol $\hat{f}$, so as to be consistent with notation from \cite{KrPi}.) 

In the next proposition,  we give a ``universal'' description of $G^*/{G^*}^{000}_{\topo}$ in the spirit of \cite[Proposition 1.20]{KrPi}, which will be used in Section \ref{section:amenability of Aut(M)}.  To prove it, we will apply Remark \ref{rem: description of G^000} and the arguments from the proof of \cite[Proposition 1.20]{KrPi}.

\begin{prop}\label{prop: universal description}
$\hat{f}$ is the (unique) initial object in the category of all maps $f \colon S_G^\mu(M) \to H$, where $H$ is a group, such that for all $p,q \in S_G(M)$, $f(\mu \cdot p)f(\mu \cdot q)$ is equal to the common value $f(\mu \cdot r)$ for all $r \in (\mu \cdot p) \cdot (\mu \cdot q)$, where the morphisms between two such maps $f_1 \colon S_G^\mu(M) \to H_1$ and $f_2 \colon S_G^\mu(M) \to H_2$ are homomorphisms from $H_1$ to $H_2$ commuting with $f_1$ and $f_2$.
\end{prop}

\begin{proof}
Uniqueness of an initial object is a general fact. 
To see that $\hat{f}$ belongs to the described category, consider any $p,q \in S_G(M)$ and any $r \in (\mu \cdot p) \cdot (\mu \cdot q)$. Then there are $a\models p' \in \mu \cdot p$ and $b\models q' \in \mu \cdot q$ such that $r=\tp(ab/M)$. Then 
$$\hat{f}(\mu \cdot p) \hat{f}(\mu \cdot q)=\hat{f}(\mu \cdot p')\hat{f}(\mu \cdot q')=(a/{G^*}^{000}_{\topo}) (b/{G^*}^{000}_{\topo})=ab/{G^*}^{000}_{\topo}=\hat{f}(\mu \cdot r).$$ 

Let us show now the universal property of $\hat{f}$.
Consider any $f\colon S^\mu_G(M) \to H$ as in the proposition. Since $\tp(e_G/M) \in (\mu \cdot \tp(e_G/M) )\cdot (\mu \cdot \tp(e_G/M))$, we have $f(\mu \cdot \tp(e_G/M))=f(\mu \cdot \tp(e_G/M)) f(\mu \cdot \tp(e_G/M))$, and so $f(\mu)=f(\mu \cdot \tp(e_G/M))=e_H$. So, if for some $a,b \in G^*$ we have $\tp(a/M)=\tp(b/M)$, then $f(\mu \cdot \tp(b^{-1}a/M))=f(\mu \cdot \tp(b^{-1}/M)) \cdot f(\mu \cdot \tp(a/M))=f(\mu \cdot \tp(b^{-1}b/M))=e_H$; in particular, $f(\mu\cdot p^{-1})=f(\mu \cdot p)^{-1}$ for all $p \in S_G(M)$, where $p^{-1}=\tp(a^{-1}/M)$ for any $a \models p$. 
Hence, for any $a \in \mu^{G^*}$ we can write $a=g^{-1}a'g$ for some $g \in G^*$ and $a' \in \mu$ and we get $f(\mu \cdot \tp(a/M))=f(\mu \cdot \tp(g^{-1}/M)) f(\mu \cdot \tp(a'/M)) f(\mu \cdot \tp(g/M))= f(\mu \cdot \tp(g/M))^{-1} f(\mu) f(\mu \cdot \tp(g/M))=e_H$.

In order to finish the proof, it is enough to show that $h\colon {G^*}/{G^*}^{000}_{\topo} \to H$ given by $h(a/{G^*}^{000}_{\topo})=f(\mu \cdot \tp(a/M))$ is a well-defined homomorphism (it is enough, because $h$ clearly commutes with $\hat{f}$ and $f$, and $\hat{f}$ is surjective). So, take $a,b \in G^*$ such that $a{G^*}^{000}_{\topo}=b{G^*}^{000}_{\topo}$. By Remark \ref{rem: description of G^000} and the well-known description (\cite[Lemma 2.2(2)]{Gis-1}) of ${G^*}^{000}_M$ as the collection of products of elements of the form $x^{-1}y$ for some $x \equiv_M y$, we get that $b^{-1}a=a_1 \cdot \ldots \cdot a_k \cdot \beta_1^{-1}\alpha_1 \cdot \ldots \cdot \beta_n^{-1}\alpha_n$ for some $a_i, \alpha_j,\beta_j \in G^*$ such that $a_i \in \mu ^{G^*}$ and $\tp(\alpha_j/M)=\tp(\beta_j/M)$ for all $i,j$. Therefore, by the last paragraph and the property of $f$, it follows that $f(\mu \cdot \tp(a/M))=f(\mu \cdot \tp(b/M))$, so we have proved that $h$ is well-defined. The fact that it is a homomorphism follows from the property of $f$.
\end{proof}

Note that equivalently one can say that the map from $S_G(M)$ to $G^*/{G^*}^{000}_{\topo}$ given by $\tp(a/M) \mapsto a/{G^*}^{000}_{\topo}$ is an initial object in the category of all maps $f \colon S_G(M) \to H$, where $H$ is a group, such that $f$ is induced by a homomorphism from $G^*$ to $H$ which is  trivial on $\mu$.

Note also that, using Proposition \ref{prop: universal description}, Corollary \ref{cor: independence of the language} and the natural identification of $G^*/E_\mu$ with $S_G^\mu(M)$, we get:

\begin{cor}\label{cor: independence of the language of G000}
${G^*}^{000}_{\topo}$ and $G^*/{G^*}^{000}_{\topo}$ (with the logic topology) do not depend on the choice of the language (assuming that predicates for all open subsets of $G$ are in the language).
\end{cor}

Using Proposition \ref{prop: characterization of G^00_top} and the definition of ${G^*}^{000}_{\topo}$, one easily checks that the quotients $G^*/{G^*}^{00}_{\topo}$ and $G^*/{G^*}^{000}_{\topo}$ do not depend as topological groups (with the logic topology) on the choice of the monster model.

In Example \ref{ex: 000}, using \cite[Corollary 0.5]{KrPi}, Corollary \ref{main corollary 3} and Example \ref{ex: Anand 1}, we will see that for $G$ and $M$ from Example \ref{ex: Anand 1}  we also have ${G^*}^{000}_M\subsetneq {G^*}^{000}_{\topo}$.


As in the case of $G^*/{G^*}^{00}_{M}$ and $G^*/{G^*}^{000}_{M}$, while $G^*/{G^*}^{00}_{\topo}$ with the logic topology is compact (we have seen in Proposition \ref{fac: from GPP} that it is even the Bohr compactification of $G$), $G^*/{G^*}^{000}_{\topo}$ is only quasi-compact (so not necessarily Hausdorff). In fact, using Proposition \ref{prop: characterization of G^00_top} and the fact that $\cl (e/{G^*}^{000}_{\topo})$ is a subgroup of $G^*/{G^*}^{000}_{\topo}$, we easily get   

\begin{rem}
$\cl (e/{G^*}^{000}_{\topo}) = {G^*}^{00}_{\topo}/{G^*}^{000}_{\topo}$.
\end{rem}

\begin{proof}
$\cl(e/{G^{*}}^{000}_{\topo})$ is by definition of the (logic) topology on $G^{*}/{G^{*}}^{000}_{\topo}$ the smallest subset of $G^{*}/{G^{*}}^{000}_{\topo}$ which contains the identity and whose preimage in $G^{*}$ is type-definable over $M$. But it is also a group, so by Proposition 2.5, the preimage is seen to be ${G^{*}}^{00}_{\topo}$ which suffices. 
\end{proof}

\noindent
Hence, the logic topology on ${G^*}^{00}_{\topo}/{G^*}^{000}_{\topo}$ is trivial, and so rather useless. Thus, a question arises, how to treat $G^*/{G^*}^{000}_{\topo}$ and ${G^*}^{00}_{\topo}/{G^*}^{000}_{\topo}$ as mathematical objects and how to measure their complexity. A possible answer in the context of $G^*/{G^*}^{000}_{M}$ was given in \cite{KrPi} (with further applications to Borel cardinalities in \cite{KrPiRz}), and below we note that the arguments from \cite{KrPi} go through also in our topological context. We do not repeat the proofs, as they are almost the same as in \cite{KrPi}. The material contained in the rest of this subsection will not be used in further sections, so the reader may skip it with no harm, but we should mention here that  Corollary \ref{main corollary 3} confirms Conjecture \ref{con: topological version of the conjecture from KrPi} under the stronger assumption of strong amenability. If the reader is interested in details standing behind the remaining part of Subsection \ref{subsection: topological category}, he or she can either go through the proofs in \cite{KrPi}, or consult Section 6.4 of the very recently written Ph.D. thesis \cite{RzePhD}, where it is briefly explained how Theorems \ref{main theorem 1} and \ref{main theorem 2} follow from a similar, but much more general, abstract result (which is fully proved in \cite{RzePhD}).


The next remark follows easily from Proposition \ref{prop: description of *}.

\begin{rem}
The function $\hat{f}$ defined after Remark \ref{rem: description of G^000} is a semigroup epimorphism.
\end{rem}

Let ${\mathcal M}$ be a minimal left ideal in the semigroup $S_G^\mu(M)$, and let $u \in \mathcal M$ be an idempotent. Then $u{\mathcal M}$ is a group called the {\em Ellis group} (of the universal ambit $S_G^\mu(M)$). Let $f=\hat{f}|_{u{\mathcal M}}: u{\mathcal M} \to G^*/{G^*}^{000}_{\topo}$. The last remark implies

\begin{rem}
$f$ is a group epimorphism.
\end{rem}

In Chapter IX of \cite{Gl} on $\tau$-topologies and the description of the generalized Bohr compactification, it is assumed ``for convenience'' that the acting group is discrete. One can check that all the material works also in the case of topological groups, after noticing the following items:

\begin{itemize}
\item $2^{S_G^\mu(M)}$ is a $G$-flow (i.e. the action of $G$ is continuous), which is easy,
\item  products of $G$-flows are $G$-flows (which is obvious),
\item quotients of $G$-flows by closed, $G$-invariant equivalence relations are $G$-flows (this is easy, see e.g. the proof of \cite[Proposition 3.7]{Kr_Polish}).
\end{itemize}
So we have a topology, called the {\em $\tau$-topology}, on $u{\mathcal M}$ which is quasi-compact, $T_1$ and such that the group operation is separately continuous. We define $H(u{\mathcal M})$ as the intersection of the $\tau$-closures of the $\tau$-neighborhoods of the identity in $u{\mathcal M}$. Then $u{\mathcal M}/H(u{\mathcal M})$ is a compact topological (Hausdorff) group (with the quotient topology induced from the $\tau$-topology). Moreover, it is the {\em generalized Bohr compactification} of $G$ in the sense of \cite{Gl}. 

Let us emphasize that we claim that Chapter IX of Glasner's book works in the topological context. In contrast, in Section 2 of our previous paper \cite{KrPi}, we had to find some new arguments in the externally definable case, since we did not know whether the $G$-flow $2^{S_{G,\ext}(M)}$ was externally definable (where $S_{G,\ext}(M)$ is the space of external types over $M$).  The argument from Section 2 of \cite{KrPi} also works in the topological case, because  the only properties that we need for that are that products of $G$-flows are $G$-flows and quotients of $G$-flows by closed, $G$-invariant equivalence relations are $G$-flows.

Recall that ${G^*}^{000}_M$ can be written as the increasing union $\bigcup_{n \in \omega} F_n$, where $F_n$ is the $M$-type-definable set consisting of products of $n$ elements of the form $b^{-1}a$ where $(a,b)$ extends to an $M$-indiscernible sequence. From Remark \ref{rem: description of G^000}, we get the following description of ${G^*}^{000}_{\topo}$.

\begin{rem}
${G^*}^{000}_{\topo}$ can be written as the increasing union $\bigcup_{n \in \omega} E_n$, where $E_n$ consists of products $xy$, where $x$ is a product of $n$ conjugates of elements of $\mu$ and $y$ is a product of $n$ elements of the form $b^{-1}a$ where $(a,b)$ extends to an infinite $M$-indiscernible sequence. 
\end{rem}

We clearly have $F_n\cdot F_m=F_{n+m}$ and $E_m \cdot E_n=E_{m+n}$. Thus, working with $S_G^\mu(M)$ in place of $S_{G,M}(N)$ (where $N$ is $|M|^+$-saturated and $S_{G,M}(N)$ is the space of complete types over $N$ concentrating on $G$ which are finitely satisfiable in $M$, i.e. the space of external types over $M$) and using $E_n$ in place of $F_n$, one can easily adapt the proof of Theorem 0.1 from \cite{KrPi} to get the following topological variant of this theorem.

\begin{thm}\label{main theorem 1}
Equip $u{\mathcal M}$ with the $\tau$-topology and $u{\mathcal M}/H(u{\mathcal M})$ -- with the induced quotient topology. Then:
\begin{enumerate}
\item $f$ is continuous,
\item $H(u{\mathcal M}) \leq \ker(f)$,
\item the formula $p/H(u{\mathcal M}) \mapsto f(p)$ yields a well-defined continuous epimorphism $\bar f$ from $u{\mathcal M}/H(u{\mathcal M})$ to $G^*/{G^*}^{000}_{\topo}$.
\end{enumerate}
In particular, we get the following sequence of continuous epimorphisms:
\begin{equation}
u{\mathcal M}\twoheadrightarrow u{\mathcal M}/H(u{\mathcal M})\xtwoheadrightarrow{\bar f}{}{G^*}/{G^*}^{000}_{\topo}\twoheadrightarrow {G^*}/{G^*}^{00}_{\topo},
\end{equation}
where $u{\mathcal M}/H(u{\mathcal M})$ is the generalized Bohr compactification of $G$.
\end{thm}


The proof of Theorem 0.2 from \cite{KrPi} also goes through without much change.

\begin{thm}\label{main theorem 2}
The group ${G^*}^{00}_{\topo}/{G^*}^{000}_{\topo}$ is isomorphic to the quotient of a compact (Hausdorff) group by a dense subgroup. More precisely, for $Y:=\ker(\bar f)$ let $\cl_\tau(Y)$ be its closure inside $u{\mathcal M}/H(u{\mathcal M})$. Then $\bar f$ restricted to $\cl_\tau(Y)$ induces an isomorphism between $\cl_\tau(Y)/Y$ (the quotient of a compact group by a dense subgroup) and ${G^*}^{00}_{\topo}/{G^*}^{000}_{\topo}$.
\end{thm}

Thus, Corollary 0.3 from \cite{KrPi} also holds in the topological context.

\begin{cor}\label{easy corollary}
i) If the epimorphism $\bar f\colon u{\mathcal M}/H(u{\mathcal M}) \to G^*/{G^*}^{000}_{\topo}$ is an isomorphism, then ${G^*}^{000}_{\topo}={G^*}^{00}_{\topo}$.\\
ii) If the epimorphism $f\colon u{\mathcal M} \to G^*/{G^*}^{000}_{\topo}$ is an isomorphism, then $H(u{\mathcal M})$ is trivial and ${G^*}^{000}_{\topo}={G^*}^{00}_{\topo}$. 
\end{cor}

Let $\xi \colon G^*/{G^*}^{000}_{\topo} \to G^*/{G^*}^{00}_{\topo}$ be the obvious map.
The proof of Corollary 0.4 from \cite{KrPi} also goes through, and we get 

\begin{cor}\label{main corollary 3}
Suppose $G$ is strongly amenable. Then the epimorphism $\xi \circ \bar f\colon u{\mathcal M}/H(u{\mathcal M}) \to G^*/{G^*}^{00}_{\topo}$ is an isomorphism, so ${G^*}^{000}_{\topo}={G^*}^{00}_{\topo}$. In particular, this holds when $G$ is nilpotent.
\end{cor}

\begin{ex}\label{ex: 000}
Let $G$ be the additive group of the reals, and let $M$ be the group $G$ expanded by predicates for all subsets of $G$. By Corollary 0.5 from \cite{KrPi}, we have ${G^*}^{000}_M = {G^*}^{00}_M$. By the previous corollary, we have ${G^*}^{000}_{\topo}={G^*}^{00}_{\topo}$. By Example \ref{ex: Anand 1}, we know that ${G^*}^{00}_M \subsetneq {G^*}^{00}_{\topo}$. Therefore, ${G^*}^{000}_M \subsetneq {G^*}^{000}_{\topo}$.
\end{ex}

A natural question is to understand for which groups ${G^*}^{000}_{\topo}={G^*}^{00}_{\topo}$. Conjecture \ref{con: topological version of the conjecture from KrPi}  strengthens Corollary \ref{main corollary 3} and predicts that it is true for amenable (in particular, for solvable) groups. In Section \ref{section: amenability of G}, we prove it for topological groups possessing a basis of open neighborhoods of the identity consisting of open subgroups. Here, we give  an example of a non-discrete topological group for which the two components are different.

Recall that the universal cover of   $\Sl_2({\mathbb R})$ can be written as the product ${\mathbb Z}\times \Sl_2({\mathbb R})$ on which multiplication is given by the standard 2-cocycle taking values $-1,0,1$  so that that the projection on the second coordinate is the covering map and there is an open neighborhood $U$ of the identity in this universal cover which is contained in $\{0\} \times  \Sl_2({\mathbb R})$ (e.g. see \cite{As}).

\begin{ex}
Let $G$ be the universal cover of $\Sl_2({\mathbb R})$ written as above. 
We treat $G$ as a group definable in any expansion $M$ of the 2-sorted structure $(({\mathbb Z},+), ({\mathbb R},+,\cdot))$ which has predicates  for all open subsets of $G$.
Then
$${G^*}^{000}_{\topo}={G^*}^{000}_M= ({{\mathbb Z}^*}^{000}+{\mathbb Z}) \times \Sl_2({\mathbb R}^*) \subsetneq G^*={G^*}^{00}_M={G^*}^{00}_{\topo},$$
where ${{\mathbb Z}^*}^{000}$ denotes the invariant connected component of ${\mathbb Z}^*$ computed in the expansion of $({\mathbb Z},+)$ by predicates for all subsets of ${\mathbb Z}$ (note that, by Corollary 0.5 from \cite{KrPi}, this coincides with the analogously defined ${{\mathbb Z}^*}^{00}$).
\end{ex}

\begin{proof}
Since the topology induced on the definable subgroup ${\mathbb Z}$ of $G$ is discrete, all subsets of ${\mathbb Z}$ are definable in the structure induced from $M$, so, modifying slightly the argument from \cite[Theorem 3.2]{CoPi}, one can show that
$${G^*}^{000}_M= ({{\mathbb Z}^*}^{000}+{\mathbb Z}) \times \Sl_2({\mathbb R}^*) \subsetneq G^*={G^*}^{00}_M.$$
Note that this already implies that ${G^*}^{000}_M$ and  ${G^*}^{00}_M$ do not depend on the choice of the language (as long as the language contains predicates for all open subsets of $G$, of course).

Since there is an open neighborhood $U$ of the identity in $G$ which is contained in $\{0\} \times  \Sl_2({\mathbb R})$, we see that
 $\mu \subseteq \{ 0 \} \times \Sl_2({\mathbb R}^*)$. But from the above description of ${G^*}^{000}_M$, $\{0\} \times  \Sl_2({\mathbb R})\subseteq {G^*}^{000}_M$. Hence, $\mu \leq {G^*}^{000}_M$, and we conclude that ${G^*}^{000}_{\topo}={G^*}^{000}_M$. The equality ${G^*}^{00}_M={G^*}^{00}_{\topo}$ follows from the fact that $G^* \geq {G^*}^{00}_{\topo} \geq {G^*}^{00}_M=G^*$.  
\end{proof}


The next corollary follows from the last example, Fact \ref{fac: from GPP} and Theorem \ref{main theorem 1}, and gives us non-trivial information about the Bohr compactification and the generalized Bohr compactification of the universal cover of $\Sl_2({\mathbb R})$ treated as a topological group.

\begin{cor}\label{cor: classical application}
The Bohr compactification of the universal cover of $\Sl_2({\mathbb R})$ is trivial, whereas its generalized Bohr compactification is non-trivial and it has as a homomorphic image the group ${\mathbb Z}^*/ ({{\mathbb Z}^*}^{00}+{\mathbb Z}) \cong ({\mathbb Z}^*/ {{\mathbb Z}^*}^{00})/ (({{\mathbb Z}^*}^{00}+{\mathbb Z})/{{\mathbb Z}^*}^{00})$ which is the Bohr compactification of the discrete group ${\mathbb Z}$ divided by a dense subgroup which is a copy of ${\mathbb Z}$.
\end{cor}

\subsection{The definable topological category}\label{subsection: def top}

Here, we define and describe connected components and the universal ambit for topological groups which are definable in arbitrary structures (not necessarily containing predicates for all open subsets of $G$), working in the ``category of definable, continuous functions and flows''.

Throughout this subsection, $G$ is a topological group which is definable in a structure $M$. Let us emphasize that we do not assume any connection between the topology on $G$ and the definable subsets of $G$. However, a special case which will concern us later is when the members of a basis of open neighborhoods of the identity are definable in $M$. 

The language of $M$ will be denoted by ${\mathcal L}$, and let ${\mathcal L'}$ be any language containing ${\mathcal L}$ and relation symbols whose interpretations in $M$ range over all open subsets of $G$. We work in $\C \succ M$ which is a monster model in the sense of both ${\mathcal L}$ and ${\mathcal L}'$. As always, $G^*$ denotes the interpretation of $G$ in $\C$. When we talk about definable sets, we mean ${\mathcal L}$-definable sets unless we say otherwise. Similarly, ${G^*}^{00}_{M}$ and ${G^*}^{000}_{M}$ are computed in ${\mathcal L}$. Also, $S_G(M)$ denotes the space of complete types in the sense of ${\mathcal L}$, and $S_G^{\mathcal L'}(M)$ -- in the sense of ${\mathcal L'}$.

We will be interested in definable, continuous functions from $G$ to compact (Hausdorff) spaces. By Section \ref{section: preliminaries}, we have

\begin{lem}\label{lem: formula for f^*}
Let $f \colon G \to C$ be a definable function, where $C$ is a compact space. Then the unique $M$-definable in ${\mathcal L}$ (and also in ${\mathcal L'}$) function $f^*\colon G^* \to C$ extending $f$ is given by 
 $$f^*(a)=\bigcap_{\varphi \in \tp_{\mathcal L}(a/M)} \cl(f[\varphi(M)])=\bigcap_{\varphi \in \tp_{\mathcal L'}(a/M)} \cl(f[\varphi(M)]).$$
If $f$ is additionally continuous, $f^*$ coincides with the map defined in Lemma \ref{lem: extending f}.
\end{lem}

\begin{dfn}
We define ${G^*}^{00}_{\defi, \topo}:={G^*}^{00}_{\topo} \cdot {G^*}^{00}_{M}$.
\end{dfn}

By the normality of ${G^*}^{00}_{\topo}$ and  ${G^*}^{00}_{M}$, ${G^*}^{00}_{\defi, \topo}$ is also a normal subgroup of $G^*$. Moreover,  ${G^*}^{00}_{\defi, \topo}$ is $M$-type-definable in ${\mathcal L'}$ and we equip $G^*/{G^*}^{00}_{\defi,\topo}$ with the logic topology computed in ${\mathcal L'}$.

\begin{prop}\label{prop: description of G^00_def,top}
i)  ${G^*}^{00}_{\defi, \topo}= \mu \cdot {G^*}^{00}_{M}$.\\
ii) ${G^*}^{00}_{\defi, \topo}$ is the smallest $M$-type-definable in ${\mathcal L}$, bounded index subgroup of $G^*$ containing ${G^*}^{00}_{\topo}$ (equivalently, containing $\mu$).\\
iii) The quotient map from $G$ to $G^*/{G^*}^{00}_{\defi,\topo}$ is a definable, continuous compactification of $G$.\\
iv) The quotient map from $G$ to $G^*/{G^*}^{00}_{\defi,\topo}$ is in fact the definable, continuous Bohr compactification of $G$ (i.e. the unique up to isomorphism universal definable, continuous compactification of $G$).
\end{prop}

\begin{proof}
i) This follows from the fact that $\mu \leq {G^*}^{00}_{\topo} = \mu \cdot {G^*}^{00}_{M, {\mathcal L'}}\leq \mu \cdot {G^*}^{00}_{M}$, where ${G^*}^{00}_{M, {\mathcal L'}}$ is computed in the language ${\mathcal L'}$, which we have by Corollary \ref{cor: formula for G^00_top}.\\[1mm]
ii) It is clear that  ${G^*}^{00}_{\defi, \topo}$ is the smallest subgroup of $G^*$ which contains  ${G^*}^{00}_{\topo}$ and ${G^*}^{00}_{M}$ (equivalently, which contains $\mu$ and ${G^*}^{00}_{M}$). Since any $M$-type-definable in ${\mathcal L}$ subgroup of $G^*$ of bounded index contains  ${G^*}^{00}_{M}$, it remains to show that ${G^*}^{00}_{\defi, \topo}$ is $M$-type-definable in ${\mathcal L}$. But this follows from the following two observations: ${G^*}^{00}_{\defi, \topo}$ is ($M$-)type-definable in ${\mathcal L'}$; ${G^*}^{00}_{\defi, \topo}$ is $M$-invariant in ${\mathcal L}$ (which is true, because $a \equiv_M b$ implies that $a^{-1}b \in {G^*}^{00}_{M} \leq {G^*}^{00}_{\defi, \topo}$).\\[1mm]
iii) The fact that this quotient map $\pi$ is a homomorphism with dense image is clear. Continuity of $\pi$ follows from the continuity of the quotient map $G \to G^*/{G^*}^{00}_{\topo}$ (by Fact \ref{fac: from GPP}) and the obvious map $G^*/{G^*}^{00}_{\topo} \to G^*/{G^*}^{00}_{\defi,\topo}$. It remains to check that $\pi$ is definable. 
For this note that the logic topologies on $G^*/{G^*}^{00}_M$ computed in ${\mathcal L}$ and ${\mathcal L'}$ coincide.
Therefore, the obvious map $G^*/{G^*}^{00}_M \to G^*/{G^*}^{00}_{\defi,\topo}$ is continuous. Since the quotient map $G \to  G^*/{G^*}^{00}_M$ is definable (by \cite[Propostion 3.4]{GPP}), we conclude that $\pi$ is definable as well.\\[1mm]
iv) Let $f\colon G \to C$ be a definable, continuous compactfication of $G$. Take $f^*\colon G^* \to C$ as in Lemma \ref{lem: formula for f^*}. By the proof of \cite[Propostion 3.4]{GPP}, we know that $f^*$ is a homomorphism and ${G^*}^{00}_M \leq \ker(f^*)$. By the last paragraph of the proof of Fact \ref{fac: from GPP}, we know that ${G^*}^{00}_{\topo} \leq \ker(f^*)$. Therefore, ${G^*}^{00}_{\defi,\topo}\leq \ker(f^*)$, and we finish as usual (see the last sentence of the proof of Fact \ref{fac: from GPP}).
\end{proof}

\begin{dfn}
We define ${G^*}^{000}_{\defi, \topo}:={G^*}^{000}_{\topo} \cdot {G^*}^{000}_{M}$.
\end{dfn}

${G^*}^{000}_{\defi, \topo}$ is clearly a normal subgroup of $G^*$ which is $M$-invariant in ${\mathcal L'}$.

\begin{rem}\label{rem: three characterizations of G^000_defi,topo}
i) ${G^*}^{000}_{\defi, \topo}$ is the smallest $M$-invariant in ${\mathcal L}$, bounded index subgroup of $G^*$ containing ${G^*}^{000}_{\topo}$.\\
ii) ${G^*}^{000}_{\defi,\topo}$ is the smallest normal, bounded index, $M$-invariant in ${\mathcal L}$ subgroup  of $G^*$ containing $\mu$.\\
iii) ${G^*}^{000}_{\defi,\topo}=\langle \mu ^{G^*}\rangle \cdot {G^*}^{000}_M$.
\end{rem}

\begin{proof}
i) Each $M$-invariant in ${\mathcal L}$ subgroup of bounded index contains ${G^*}^{000}_{M}$, so it is enough to show ${G^*}^{000}_{\defi, \topo}$ is $M$-invariant in ${\mathcal L}$, which follows as in the proof of Proposition \ref{prop: description of G^00_def,top}(ii).\\
ii) follows easily from (i) and  the definitions of ${G^*}^{000}_{\defi,\topo}$ and ${G^*}^{000}_{\topo}$.\\
iii) follows from (ii) and the $M$-invariance  in ${\mathcal L}$ of $\langle \mu ^{G^*}\rangle \cdot {G^*}^{000}_M$.
\end{proof}

\begin{rem}\label{rem: equality of various components}
i) If $G$ is discrete, then ${G^*}^{000}_{\defi, \topo}={G^*}^{000}_{M}\geq {G^*}^{000}_{\topo}$ and ${G^*}^{00}_{\defi, \topo}={G^*}^{00}_{M}\geq {G^*}^{00}_{\topo}$.\\
ii) If ${\mathcal L}$ contains predicates for all open subsets of $G$, then ${G^*}^{000}_{\defi, \topo}={G^*}^{000}_{\topo} \geq {G^*}^{000}_{M}$ and ${G^*}^{00}_{\defi, \topo}={G^*}^{00}_{\topo} \geq {G^*}^{00}_{M}$.
\end{rem}

\begin{rem}
Note that the characterizations from Proposition \ref{prop: description of G^00_def,top}(i) and Remark \ref{rem: three characterizations of G^000_defi,topo}(iii) can be used as definitions of  ${G^*}^{00}_{\defi,\topo}$ and ${G^*}^{000}_{\defi,\topo}$, even in the wider context when ${\mathcal L'}$ is any extension of ${\mathcal L}$ such that all members of some basis of open neighborhoods of the identity in $G$ are definable in ${\mathcal L'}$ (with parameters from $M$); the difference is that now more monster models are allowed, because we do not require ${\mathcal L'}$ to contain predicates for all open subsets of $G$. Then, by a standard argument, we get that the quotients $G^*/{G^*}^{00}_{\defi, \topo}$ and $G^*/{G^*}^{000}_{\defi, \topo}$ do not depend on the choice of the monster model in which they are computed.
\end{rem}

Our next goal is to give a description of the {\em universal definable topological $G$-ambit}, i.e. the universal $G$-ambit in the category of $G$-ambits which are both definable and topological. Recall that in \cite{Kr} the universal definable $G$-ambit (of $G$ treated as a discrete group) was described as the quotient $S_G(M)/E$, or equivalently as $G^*/E'$, for a certain closed equivalence relation $E$ on $S_G(M)$ and the corresponding $M$-type-definable equivalence relation $E'$ on $G^*$. 
(A description of the relation $E'$ can be found in Section 2 of \cite{Kr}. We do not recall this description here, because we will not use it.)
We also consider the relation $E_\mu$ defined in the previous subsection, so that $G^*/E_\mu$ is the universal (topological) $G$-ambit.

Define $E_1'$ to be the finest $M$-type-definable in ${\mathcal L'}$ equivalence relation on $G^*$ which contains $E' \cup E_\mu$. By Remark \ref{rem: E_mu as composition} and the fact that $\equiv_M^{\mathcal L'}\;\subseteq \;\equiv_M^{\mathcal L}\; \subseteq E'$,
we see that $E_1'$ is the finest $M$-type-definable in ${\mathcal L'}$ equivalence relation on $G^*$ which contains $E' \;\cup \sim$, where $\sim$ is the relation defined in the previous subsection.

\begin{rem}\label{rem: description of E_1'}
$E_1'$ is the finest $M$-type-definable in ${\mathcal L}$ equivalence relation on $G^*$ which contains $E' \cup E_\mu$ (equivalently, $E' \;\cup \sim$).
\end{rem}

\begin{proof}
This follows easily from the observation that the logic topologies on $G^*/E'$ computed in ${\mathcal L}$ and in ${\mathcal L'}$ coincide. Indeed, from this, the obvious map $\pi \colon G^*/E' \to G^*/E_1'$ is continuous. Therefore, it is easy to see that $E_1'$, being the preimage of the diagonal by the obvious map $G^* \times G^* \to G^*/E_1' \times G^*/E_1'$, must be $M$-type-definable in ${\mathcal L}$.
\end{proof}

Let $E_1$ be the equivalence relation on $S_G(M)$ given by 
$$pE_1 q \iff (\exists a \models p) (\exists b \models q) (a E_1' b) \iff (\forall a \models p) (\forall b \models q) (a E_1' b).$$
We leave to the reader to check that $E_1$ is the finest closed equivalence relation on $S_G(M)$ which contains both $E$ and the (not necessarily equivalence) relation $\sim_\mu^{\mathcal L}$ given by $p\sim_\mu^{\mathcal L}q \iff  (\exists a \models p) (\exists b \models q) (ab^{-1} \in \mu)$.

Since $E'$ and $E_\mu$ are both $G$-invariant, we get that that $G$ acts on $G^*/E_1'$ by $g(a/E_1')=(ga)/E_1'$ and on $S_G(M)/E_1$ by $g(\tp(a/M)/E_1)=\tp(ga/M)/E_1$. The assignment $a/E_1' \mapsto \tp(a/M)/E_1$ is a homeomorphism from $G^*/E_1'$ to $S_G(M)/E_1$ preserving the action of $G$.

\begin{rem}
$(G, G^*/E_1', e/E_1')$ (equivalently, $(G,S_G(M)/E_1,\tp(e/M)/E_1)$) is a definable topological $G$-ambit.
\end{rem}

\begin{proof}
We have the obvious continuous map map $\theta \colon G^*/E_\mu \to G^*/E_1'$ which preserves the action of $G$. Hence, since $G^*/E_\mu$ and $G^*/E_1'$ are compact, continuity of the action of $G$ on $G^*/E_1'$ follows from the continuity of the action of $G$ on $G^*/E_\mu$.

Now, we check the definability of the ambit. Take any $a/E_1' \in G^*/E_1'$, and let $f_a^1 \colon G \to G^*/E_1'$ be given by $g \mapsto g(a/E_1')$ and $f_a \colon G \to G^*/E'$ by $g \mapsto g(a/E')$. We have the obvious continuous map $\eta \colon G^*/E' \to G^*/E_1'$ such that $f_a^1=\eta \circ f_a$. Therefore,  since we know that $f_a$ is definable, $f_a^1$ is also definable.
\end{proof}

\begin{prop}\label{prop: description of the universal definable topological ambit}
$(G, G^*/E_1', e/E_1')$ (equivalently, $(G,S_G(M)/E_1,\tp(e/M)/E_1)$) is the universal definable topological $G$-ambit.
\end{prop}

\begin{proof}
Let $(G,X,x)$ be an arbitrary definable topological $G$-ambit. Define $f_x \colon G \to X$ by $f_x(g)=gx$; it is continuous and definable. Take the extension $f_x^* \colon G^* \to X$ of $f_x$ given by Lemma \ref{lem: formula for f^*}. From the explicit formula for $f_x^*$, we see that it preserves the action of $G$. If we show that $f_x^*$ factors through $E_1'$, we will get a homomorphism from the ambit $(G,G^*/E_1',e/E_1')$ to $(G, X,x)$, and the proof will be complete. In order to get this factorization, it is enough to show that $f_x^*$ factors through both $E_\mu$ and $E'$. Factorization through $E_\mu$ was explicitly proved in the proof of Fact \ref{fac: universal G-ambit}, so it remains to show factorization through $E'$. 

Let $h_x \colon S_G(M) \to X$ be the factorization of $f_x^*$. This is a unique homomorphism from the $G$-ambit $(G,S_G(M),\tp(e/M))$ (for $G$ treated as a discrete group) to the $G$-ambit $(G,X,x)$. On the other hand, by the universality of the definable (not topological) $G$-ambit $(G,S_G(M)/E,\tp(e/M)/E)$, we get a unique homomorphism $k_x$ from $(G,S_G(M)/E,\tp(e/M)/E)$ to $(G,X,x)$. This induces a homomorphism $\tilde{k}_x$ from $(G,S_G(M),\tp(e/M))$ to $(G,X,x)$. By the uniqueness of $h_x$, we get $\tilde{k}_x=h_x$. Thus, $h_x$ factors through $E$, which implies that $f_x^*$ factors through $E'$.
\end{proof}

We have the following obvious epimorphisms of $G$-ambits (recall that $S_G(M)$ and $S_G(M)/E$ are $G$-ambits for $G$ considered as a discrete group; the others are topological $G$-ambits).

\begin{figure}[H]
		\centering
		\begin{tikzcd}
			& S_G(M)\arrow[dl, two heads]\arrow[dd, two heads ] & S_G^{{\mathcal L'},\mu}(M) \arrow[ddl, two heads]\\
			S_G(M)/E\arrow[dr, two heads]  & & \\
			&  S_G(M)/E_1 & 	
		\end{tikzcd}
	\end{figure}

We discuss here the special case when there is a basis of open neighborhoods of the identity in $G$ consisting of sets which are definable in the language ${\mathcal L}$ (with parameters from $M$).
 In such a situation, let $S_G^{\mu}(M)$ be the quotient $S_G(M)/\!\!\sim_\mu$, where $\sim_\mu$ is the equivalence relation on $S_G(M)$ defined by
$p \sim_\mu q \iff \mu \cdot p=\mu \cdot q$. Claim A.5 in \cite{PeSt} says that $(G,S_G^{\mu}(M),\tp(e/M)/\!\!\!\sim_\mu)$ is a topological $G$-ambit, where $g\cdot (p/\!\!\sim_\mu):=(gp)/\!\!\sim_\mu$. In fact, the whole discussion between Example \ref{ex: Anand 1} and Fact \ref{fct: continuity of the ambit} (including this fact) goes through in the present context. However, in general, this ambit does not have to be definable and it is not universal in any of our categories.
We have the following natural epimorphisms of $G$-ambits ($S_G(M)$ and $S_G(M)/E$ are $G$-ambits for $G$ considered as a discrete group; the others are topological $G$-ambits).

\begin{figure}[H]
		\centering
		\begin{tikzcd}
			& S_G(M)\arrow[dl, two heads]\arrow[d, two heads ] & S_G^{{\mathcal L'},\mu}(M) \arrow[dl, two heads]\\
			S_G(M)/E\arrow[dr, two heads] & S_G^{\mu}(M)\arrow[d,two heads]&  \\
			& S_G(M)/E_1 & 	
		\end{tikzcd}
	\end{figure}

\begin{rem}\label{rem: ambit from PeSt in the discrete case}
i) If $G$ is discrete, then $\mu = \{ e\}$, $E_1=E$ and $S_G^{\mu}(M)=S_G(M)$.\\
ii) If ${\mathcal L}$ contains predicates for all open subsets of $G$, then the above epimorphism $S_G^{{\mathcal L'},\mu}(M) \to  S_G^{\mu}(M)$ is an isomorphism.\\
iii) If all types in $S_G(M)$ are definable (and there is a basis of open neighborhoods of the identity consisting of definable sets), then $E_1=\,\sim_\mu$ and $S_G^{\mu}(M)= S_G(M)/E_1$.
\end{rem} 

\begin{proof}
Item (i) is obvious. (ii) follows from Fact \ref{fac: universal G-ambit}. The last item follows easily from Remark 2.4(ii) in \cite{Kr} which says that if all types in $S_G(M)$ are definable, then the equivalence relation $E$ is trivial.
\end{proof}

\section{Amenability and connected components}\label{section: amenability of G}

\subsection{Variants of amenability}\label{subsection: variants amenability}

Recall that a topological group $G$ is said to be {\em amenable}, if for every $G$-flow $(G,X)$ there is a left-invariant, Borel probability measure on the compact space $X$; equivalently, if there is such a measure on the universal (topological) $G$-ambit. If $G$ is discrete, this is equivalent to the existence of a left-invariant, finitely additive probability measure on all subsets of $G$.

A definable in $M$ (discrete) group $G$ is {\em definably amenable} if there is a left-invariant Keisler measure on $G$ (i.e. finitely additive probability measure on the Boolean algebra of definable subsets of $G$); equivalently, if there is a left-invariant, regular, Borel probability measure on the compact space $S_G(M)$ (see \cite{Si} for details).

Working in the category of definable flows, it makes sense to define a weaker notion of definable amenability, namely we say that $G$ is {\em weakly definably amenable} if there exists a left-invariant, Borel probability measure on the universal definable $G$-ambit, i.e. on $S_G(M)/E$ using the notation from Subsection \ref{subsection: def top}. It agrees with definable amenability if all types in $S_G(M)$ are definable, because then $E$ is trivial.

Working in the definable topological category, we introduce the following notions of amenability.

\begin{dfn}\label{definition: weak definable topological amenability}
Let $G$ be a topological group definable in a structure $M$.\\ 
i) Then $G$ is {\em weakly definably topologically amenable} if there exists a left-invariant, Borel probability measure on the universal definable topological $G$-ambit, i.e. on $S_G(M)/E_1$ (using the notation from Subsection \ref{subsection: def top}).\\
ii) 
If $G$ has a basis of open neighborhoods of the identity consisting of definable sets, we say that $G$ is {\em definably topologically amenable} if there exists a left-invariant, Borel probability measure on the $G$-ambit $S_G^{\mu}(M)$ (defined before Remark \ref{rem: ambit from PeSt in the discrete case}).
\end{dfn}

The next remark follows from the diagrams in the final part of Subsection \ref{subsection: def top}.

\begin{rem}\label{rem: relationships between amenability}
Let $G$ be a topological group definable in a structure $M$.\\
i) 
If $G$ has a basis of open neighborhoods of the identity consisting of definable sets, then definable topological amenability implies weak definable topological amenability.\\
ii) Each of the conditions ``$G$ is amenable'' and ``$G$ is definably amenable'' implies ``$G$ is weakly definably topologically amenable''. In the case when 
$G$ has a basis of open neighborhoods of the identity consisting of definable sets, each of these conditions implies ``$G$ is definably topologically amenable''.\\
iii) If $G$ is discrete, then 
definable amenability is equivalent to definable topological amenability (note that $\{\{e\}\}$ is a basis at $e$ consisting of a definable set).\\
iv) If the language contains predicates for all open subsets of $G$, then amenability of $G$ as a topological group is equivalent to definable topological amenability.
\end{rem}


We finish with a justification of the relationships between the conjectures formulated in the introduction.

Conjecture \ref{con: the most general} implies Conjecture \ref{con: general enough} by Remark \ref{rem: relationships between amenability} (i). To see that Conjecture \ref{con: general enough} implies Conjectures \ref{con: conjecture from KrPi} and \ref{con: topological version of the conjecture from KrPi}, one should use Remark \ref{rem: relationships between amenability} (iii) and (iv) together with Remark \ref{rem: equality of various components}. The fact that Theorem \ref{thm: main theorem in Section 2} implies Conjecture \ref{con: conjecture from KrPi} in its full generality and Conjecture \ref{con: topological version of the conjecture from KrPi} for groups possessing a basis of open neighborhoods of $e$ consisting of open subgroups follows by the same reasons.

\subsection{Extreme amenability}\label{subsection: extreme amenability}

As a warm up case, we first study connected components for extremely amenable groups and give a quick proof of Proposition \ref{prop: extremely amenable}.

To obtain the notions of extreme amenability in the various contexts, one has to take the appropriate definitions of amenability (from the last subsection) and replace the existence of an appropriate invariant measure by the existence of a fixed point. For example, if $G$ is a topological group definable in the structure $M$ so that there is a basis of open neighborhoods of the identity consisting of definable sets, we say that $G$ is {\em definably topologically extremely amenable} if the  $G$-ambit $S_G^{\mu}(M)$ has a fixed point.

One can formulate variants of Conjectures \ref{con: conjecture from KrPi}, \ref{con: topological version of the conjecture from KrPi}, \ref{con: the most general}, \ref{con: general enough} by strengthening the amenability assumption to extreme amenability and by strenthening the conclusions to the statements that both connected components in question are equal to $G^*$. Proposition \ref{prop: extremely amenable} contains such variants of  Conjectures \ref{con: general enough}, \ref{con: topological version of the conjecture from KrPi}, and \ref{con: conjecture from KrPi}.

Let $G$ be a topological group definable in the structure $M$, such that there is a basis of open neighborhoods of the identity consisting of definable sets. 
Recall that for any $r \in S_G(M)$ the equivalence class $r/\!\! \sim_{\mu} \in S_G^\mu(M)$ consists of all complete types over $M$ extending the partial type $\mu \cdot r$; so we will freely identify $r/\!\!\sim_{\mu}$ with $\mu \cdot r$ and with the corresponding type-definable set; so, $p \in S_G^{\mu}(M)$ can be viewed as an equivalence class of $\sim_\mu$ or a partial type or the corresponding type-definable set. Similarly, types in $S_G(M)$ are identified with the corresponding type-definable sets.


\begin{proof}[Proof of Proposition \ref{prop: extremely amenable}]
By assumption, there is a $G$-invariant $p =\mu \cdot r \in S_G^{\mu}(M)$. Since $\mu  \subseteq {G^*}^{000}_{\topo} \subseteq {G^*}^{000}_{\defi, \topo}$ and $rr^{-1} \subseteq {G^*}^{000}_{M} \subseteq {G^*}^{000}_{\defi, \topo}$, we get $pp^{-1} \subseteq {G^*}^{000}_{\defi, \topo}$. So, it remains to show that $pp^{-1}=G^*$. Take any $\varphi$ over $M$ such that $p \vdash \varphi(x)$. It is enough to show that $\varphi(G^*)\varphi(G^*)^{-1}=G^*$, and for this that $\varphi(G)\varphi(G)^{-1}=G$. Consider any $g \in G$ and $a \models p$. Then $b:=ga \models p$, so $g=ba^{-1} \in pp^{-1}$, and we are done.

Now, the two additional statements follow easily from the obvious counterpart of Remark \ref{rem: relationships between amenability} (iii) and (iv) in the extremely amenable case and from Remark \ref{rem: equality of various components}.
\end{proof}

\subsection{Amenability and connected components}\label{subsection: amenability and components}

The structure of this subsection is the following. 
First, we generalize Construction $(*)$ from \cite{HPP} which yields extensions of measures.
Then we prove Conjecture \ref{con: conjecture from KrPi} using  \cite[Theorem 12]{MaWa}. In the course of the proof, we distinguish the special case when 
all subsets of $G$ are definable in which we explain how to prove this conjecture via a simplification of the proof of \cite[Theorem 12]{MaWa}. Next, after some preparatory results, we adapt this simplification of the argument from \cite{MaWa} together with the aforementioned generalization of Construction $(*)$ in order to show Conjecture \ref{con: topological version of the conjecture from KrPi} for groups possessing a basis of open neighborhoods of the identity consisting of open subgroups. Finally, we adapt the full proof of  \cite[Theorem 12]{MaWa} to get Theorem \ref{thm: main theorem in Section 2} -- the most general result of this subsection. But this argument contains some more delicate points.

The reader is asked to read first the proof of Theorem 12 from \cite{MaWa}, as we are not going to repeat all the details from there. We will explain in details the ingredients which are new in comparison with the proof from \cite{MaWa}.

We start from an elaboration on Construction $(*)$ from \cite{HPP} on extending measures to saturated models, as it will play an important role in the proofs below. Recall that a Keisler measure ${\mathfrak m}$ on a definable subset $X$ of a structure $M$ can be thought of as a collection of functions ${\mathfrak m}_\varphi \colon S_\varphi \to [0,1]$, where $\varphi(x,y)$ ranges over formulas without parameters such that $x$ is always from the sort of $X$ and $S_\varphi$ is the sort of the variable $y$ of the given formula $\varphi(x,y)$, which satisfies certain properties corresponding to the definition of a measure; more precisely, ${\mathfrak m}_\varphi (a)={\mathfrak m}(\varphi(X,a))$. We will be also interested in the situation when ${\mathfrak m}$ is a measure defined only on some Boolean subalgebra ${\mathcal A}$ of the algebra of all definable subsets $X$. In this case, add an extra element $\infty$ greater than all element of $[0,1]$, and let ${\mathfrak m}_\varphi (a)$ be equal to ${\mathfrak m}(\varphi(X,a))$ if ${\mathfrak m}(\varphi(X,a))$ is defined and $\infty$ otherwise.\\[3mm]
{\bf Generalization of Construction (*) from \cite{HPP}}.
Let ${\mathfrak m}$ be a finitely additive probability measure defined on a Boolean subalgebra ${\mathcal A}$ of the algebra of all definable subsets of $X$, where $X$ is a definable set in a structure $M$. Consider the structure $M':=(M,[0,1] \cup \{ \infty\},+,<,{\mathfrak m}_\varphi)_\varphi$ consisting of the structure $M$, the functions ${\mathfrak m}_\varphi \colon S_\varphi \to [0,1] \cup \{ \infty \}$, the ordering $<$ on $[0,1] \cup \{ \infty \}$, and addition (modulo 1) on $[0,1]$. Take a monster model $(M^*,[0,1]^*\cup \{\infty\},+,<,{\mathfrak m}_{\varphi}^*)_\varphi$. For a definable subset $\varphi(X^*,a)$ of $X^*$ put ${\mathfrak m}^*(\varphi(X^*,a))={\mathfrak m}^*_\varphi(a)$. We will say that {\em ${\mathfrak m}^*$ is defined on $\varphi(X^*,a)$} if ${\mathfrak m}^*_\varphi(a) \in [0,1]^*$ (equivalently, ${\mathfrak m}^*_\varphi(a) < \infty$). Then:

\begin{enumerate}
\item The collection of sets on which ${\mathfrak m}^*$ is defined forms a Boolean algebra of definable (in the original language) subsets of $X^*$, which we denote by ${\mathcal A}^*$.
\item The sets definable (in the original language) over $M$ which belong to ${\mathcal A}^*$ are exactly the sets coming from ${\mathcal A}$ (i.e. the interpretations in $M^*$ of the sets from ${\mathcal A}$).
\item ${\mathfrak m}^*$ is a non-standard (i.e. taking values in $[0,1]^*$) finitely additive probability measure on ${\mathcal A}^*$.
\item For any $r \in [0,1]$ and formula $\varphi(x,y)$, $\{a \in M^*: {\mathfrak m}^*(\varphi(X^*,a)) \in [0,r]^*\}$ is definable over $M'$ in the expanded language; similarly for $[r,1]^*$ in place of $[0,r]^*$. After naming parameters from $[0,1]$, it is definable over $\emptyset$.
\item If $X=G$ is a definable group and ${\mathcal A}$ is closed under products (if $X,Y \in {\mathcal A}$, the $XY\in {\mathcal A}$) or under multiplication by elements of $G$ (if $X \in {\mathcal A}$, then $gX \in {\mathcal A}$ for all $g \in G$), then so is ${\mathcal A}^*$ (with $G^*$ in place of $G$).
\item If $X=G$ is a definable group and ${\mathfrak m}$ is $G$-invariant (in particular, ${\mathcal A}$ is $G$-invariant), then ${\mathfrak m}^*$ is $G^*$-invariant.
\end{enumerate}

Now, we can define a standard (i.e. with values in $[0,1]$) finitely additive probability measure on ${\mathcal A}^*$ as the composition $\st \circ \,{\mathfrak m}^*$, where $\st \colon [0,1]^* \to [0,1]$ is the standard part map. This measure extends ${\mathfrak m}$ (i.e. coincides with ${\mathfrak m}$ on sets definable over $M$), and we will still denote it by ${\mathfrak m}$. Then:

\begin{enumerate}
\item This extended ${\mathfrak m}$ is definable over $M'$ in the sense that for any closed subset $C$ of $[0,1]$ and any formula $\varphi(x,y)$, the set $\{a \in M^*: {\mathfrak m}(\varphi(X^*,a)) \in C$\} is type-definable over $M'$ in the expanded language. After naming parameters from $[0,1]$, it is definable over $\emptyset$.
\item If $X=G$ is a definable group and the original ${\mathfrak m}$ is $G$-invariant, then the extended ${\mathfrak m}$ is $G^*$-invariant.
\end{enumerate}

Now, we will prove Conjecture \ref{con: conjecture from KrPi}.

\begin{thm}\label{thm: conjecture from KrPi}
Let $G$ be a group definable in a structure $M$. If $G$ is definably amenable, then ${G^*}^{00}_M={G^*}^{000}_M$.
\end{thm}

\begin{proof}
Let ${\mathfrak m}$ be the left-invariant Keisler measure on $G$ witnessing definable
amenability. Recall that a type $q \in S_G(M)$ is called {\em weakly ${\mathfrak m}$-random} if for any $\varphi(x) \in q$, ${\mathfrak m}(\varphi(M))>0$. Weakly ${\mathfrak m}$-random types always exist. We will identify complete types over $M$ with their sets of realizations in the monster model. Since for any $q \in S_G(M)$ we have $qq^{-1} \subseteq {G^*}^{000}_M$, in order to prove our theorem, it is enough to show

\begin{lem}\label{lem: Proposition 0.2 from Anand's notes} 
Suppose $q \in S_G(M)$ is weakly ${\mathfrak m}$-random. Then ${G^*}^{00}_M \subseteq (qq^{-1})^4$.
\end{lem}

To prove this lemma, we will prove another lemma which is an adaptation of the proof of Theorem 12 from \cite{MaWa}. The latter paper appears to need that the model $M$ is $\omega^+$-saturated, but we will get around it. Recall that a subset $X$ of $G$ is {\em symmetric} if it contains $e$ and $X^{-1}=X$; it is {\em (left) generic} if finitely many (left) 	translates of $X$ by elements of $G$ cover $G$.

\begin{lem}\label{lem: main lemma in Section 2}
Let $B$ be any definable subset of $G$ with ${\mathfrak m}(B) > 0$. Then there are symmetric,
generic, definable subsets $C_1 \supseteq  C_2 \supseteq \ldots$ of $G$ such that $C_1 \subseteq (BB^{-1})^4$ and
$C_{i+1}^2 \subseteq  C_i$ for all $i$.
\end{lem}

Let us see first how Lemma \ref{lem: Proposition 0.2 from Anand's notes} follows from Lemma \ref{lem: main lemma in Section 2}. Consider
$\bigcap_i C_i^*$ (as usual $C_i^*$ is $C_i$ computed in the monster model). It is easy to see that it is an $M$-type-definable
subgroup of $G^*$ of bounded index, hence contains ${G^*}^{00}_M$.
So, by Lemma \ref{lem: main lemma in Section 2}, $(B^*{B^*}^{-1})^4$ contains ${G^*}^{00}_M$. As $B$ was an arbitrary definable subset
of $G$ of positive measure, and since for $q \in S_G(M)$ we have $(qq^{-1})^4 = \bigcap \{(\varphi(G^*)\varphi(G^*)^{-1})^4 :
\varphi \in q\}$, it follows that $(qq^{-1})^4$ contains ${G^*}^{00}_M$ for any weakly ${\mathfrak m}$-random $q$.

\begin{proof}[Proof of Lemma \ref{lem: main lemma in Section 2}]
Define $A_1 = BB^{-1}$. Then $A_1$ is definable, symmetric and generic (by Ruzsa's covering lemma, see \cite[Fact 5]{MaWa}). 

\begin{clm*}\label{clm: from Frank}
For any definable $A\subseteq G$ which is generic and symmetric there is a generic and symmetric definable set $X\subseteq G$ such that $X^8 \subseteq A^4$.
\end{clm*}

To finish the proof using this claim, we apply it to $A:=A_1=BB^{-1}$ and we get $C_1:=A^4$ and $C_2:=X^4$. Then we apply the claim to $A:=X$ and we get an appropriate $C_3$, and so on.

Claim \ref{clm: from Frank} will be easily deduced from Theorem 12 of \cite{MaWa}.
But first, we will consider the special case when all subsets of $G$ are definable and we will briefly explain how Claim \ref{clm: from Frank} can be obtained in this case by a simplification of the proof of \cite[Theorem 12]{MaWa}, which will be later adapted to prove Theorem \ref{thm:  Topological conjecture}.  The simplification in our special case is that we do not have to use the conditions $P_n^t$ and the sets $X_n$ from the proof from \cite{MaWa}, as all subsets of $G$ are now definable. 
Note that in this case ${G^*}^{00}_M = {G^*}^{00}_{\topo}$ and ${G^*}^{000}_M = {G^*}^{000}_{\topo}$, as $G$ is considered as a discrete group. Hence, Theorem \ref{thm: conjecture from KrPi} in this special case already shows that if a group $G$ is amenable (as a discrete group), then its Bohr compactification ${G^*}/{G^*}^{00}_{\topo}$ coincides with its ``weak Bohr compactification'' ${G^*}/{G^*}^{000}_{\topo}$.

\begin{proof}[Proof of Claim \ref{clm: from Frank}] First, consider the above special case.\\

{\bf Case 1 -- All subsets of $G$ are definable in $M$.}\\
%
%
%
Let $K$ be the number of translates of $A$ needed to cover $G$. In particular,  $K$ translates of $A$ cover $A^2$,  so $A$ is a $K$-approximate subgroup as in \cite[Theorem 12]{MaWa}. Take any natural number $m>0$.
For $t \in (0, 1]$ define ${\mathcal B}_t$ as the set of
subsets $B'$ of $A$ such that ${\mathfrak m}(B') \geq t {\mathfrak m}(A)$; it is nonempty, as it contains $A$. Let $f(t) = \inf\{{\mathfrak m}(B'A)/{\mathfrak m}(A) :
B' \in {\mathcal  B}_t\}$. Fix $\epsilon>0$, and let by Sanders' Lemma (Lemma 11 of \cite{MaWa}) $t$ be such
that $f(t^2/2K) \geq (1 - \epsilon)f(t)$. Choose $B' \in {\mathcal B}_t$ with ${\mathfrak m}(B'A)/{\mathfrak m}(A) \leq (1 +
\epsilon)f(t)$. Let $X = \{g \in A^2 : {\mathfrak m}(gB' \cap B') \geq (t^2/2K){\mathfrak m}(A)\}$. 
The computation on lines 15-23 of page 61 in \cite{MaWa} (which we recall below) shows that if $g_1,\dots, g_m \in X$,
then ${\mathfrak m}(g_1\dots g_mB'A \triangle B'A) < 4m\epsilon {\mathfrak m}(B'A)$. In particular, if $\epsilon \leq 1/4m$ (even $\epsilon < 1/2m$ is enough), then
$g_1\dots g_m B'A \cap B'A$ is nonempty, whereby $X^m \subseteq A^4$. Choosing $m=8$ and $\epsilon \leq 1/32$, we
see that $X^8 \leq  A^4$. 
The statement formulated at the very beginning of the proof of Theorem 12 of \cite{MaWa} tells us that
 $X$ is generic in $G$, and it is also symmetric (because $A$ is symmetric, $B' \in  {\mathcal  B}_t \subseteq  {\mathcal  B}_{t^2/2K}$ and ${\mathfrak m}$ is invariant), so the proof in Case 1 is finished. 

As promised, for the reader's convenience we recall now  the computation from \cite{MaWa} whose conclusion was used in the above argument. For $g \in X$, ${\mathfrak m}(gB'A \cap B'A) \geq {\mathfrak m}((gB' \cap B')A) \geq f(t^2/2K){\mathfrak m}(A) \geq (1 - \epsilon)f(t){\mathfrak m}(A) \geq \frac{1 - \epsilon}{1 + \epsilon} {\mathfrak m}(B'A)$; hence, ${\mathfrak m}(gB'A \vartriangle B'A) \leq  \frac{4\epsilon}{1 + \epsilon} {\mathfrak m}(B'A) < 4 \epsilon {\mathfrak m}(B'A)$. It follows that for $g_1, \dots , g_m \in X$,
${\mathfrak m}(g_1 \dots g_mB'A \vartriangle B'A) \leq {\mathfrak m}((B'A \vartriangle g_1B'A) \cup g_1(B'A\vartriangle g_2B'A) \cup \dots \cup g_1 \dots g_{m-1}(B'A\vartriangle g_mB'A)) \leq {\mathfrak m}(B'A \vartriangle g_1B'A) + {\mathfrak m}(B'A\vartriangle g_2B'A) + \dots + {\mathfrak m}(B'A\vartriangle g_mB'A) < 4m\epsilon {\mathfrak m}(B'A)$.
\\

%
{\bf Case 2 -- General case.}\\
So here, $M$ is an arbitrary structure. 
By Construction $(*)$ described before Theorem \ref{thm: conjecture from KrPi}, we know that ${\mathfrak m}$ extends to an invariant Keisler measure (which we also call ${\mathfrak m}$) on $G^*$. 
Now, we apply \cite[Theorem 12]{MaWa} (whose proof uses the $P_n^t$'s and $X_n$'s), where we take $A$
there to be our $A^*$ and work in the monster model.  As a result we obtain a definable set $S_a:=\varphi(G^*,a)$ (where $\varphi(x,y)$ is a formula without parameters and $a$ is a tuple from the monster model) which is symmetric and generic and such that $S_a ^8 \subseteq {A^*}^4$. Let $L$ be the number of translates of $S_a$ needed to cover $G^*$. For any $b$, let $S_b =\varphi(G^*,b)$.  The following conditions on $y$

\begin{itemize}
\item $S_y$ is symmetric,
\item $L$ translates of $S_y$ cover $G^*$,
\item $S_y^8 \subseteq {A^*}^4$
\end{itemize}
are all definable over $M$. Since $y:=a$ satisfies these conditions, we can find $a' \in M$ which also satisfies them. Then $X:=\varphi(G,a')$ satisfies the requirements of Claim \ref{clm: from Frank}. 
\end{proof}
As was noted before, Claim  \ref{clm: from Frank} implies Lemma \ref{lem: main lemma in Section 2}. 
\end{proof}

The proof of Theorem \ref{thm: conjecture from KrPi} has been completed.
\end{proof}

Now, we turn to the definable topological context from Theorem \ref{thm: main theorem in Section 2}. Namely, from now on, in this section we assume that $G$ is a topological group definable in a structure $M$ so that there is a basis of open neighborhoods of the identity consisting of definable, open subgroups. The main reason why we need to assume that there is such a basis is the next remark and proposition which will allow us to define a measure on certain subsets of $G$. 

Recall that the whole discussion between Example \ref{ex: Anand 1} and Fact \ref{fct: continuity of the ambit} (including this fact) goes through in our context, and take the notation from there. In particular, $S_G^{\mu}(M) \approx G^*/E_\mu$, and we will identify these $G$-ambits. Recall that $E_\mu=\sim \circ \equiv_M$, where $a \sim b \iff ab^{-1} \in \mu$. As before, $p \in S_G^{\mu}(M)$ will be understood as an equivalence class of $\sim_\mu$ or a partial type or the corresponding type-definable set.

\begin{rem}\label{rem: basic on clopens}
Let $C \subseteq G^*/E_\mu$ and let $h\colon G^* \to G^*/E_\mu$ be the quotient map. The following conditions are equivalent.\\
i) $C$ is clopen.\\
ii) $\mu \cdot h^{-1}[C]=h^{-1}[C]$ is definable (over $M$).\\
iii) $C=X^*/E_\mu$ for some $M$-definable subset $X$ of $G$ satisfying $\mu \cdot X^*=X^*$; in this situation, we have $h^{-1}[C]=X^*$.\\
iv) $H^* \cdot h^{-1}[C]=h^{-1}[C]$ for some definable, open subgroup $H$ of $G$.\\
v) $C=(H^* \cdot X^*)/E_\mu$ for some definable, open subgroup $H$ of $G$ and $M$-definable subset $X$ of $G$; in this situation, we have $h^{-1}[C]=H^* \cdot X^*$.
\end{rem}

\begin{proof}
(i) $\rightarrow$ (ii) follows from the fact that $\sim \, \subseteq E_\mu$ and the fact that the preimage under $h$ of a clopen set is definable over $M$.\\
(ii) $\rightarrow$ (iii) is clear taking $X:=h^{-1}[C]$. For the additional claim, we see that $X^*$ is invariant under multiplication by $\mu$ on the left and under automorphisms over $M$, so it is a union of $E_\mu$-classes, and hence $h^{-1}[C]=X^*$.\\
(iii) $\rightarrow$ (iv). Assume (iii). Then $h^{-1}[C]=X^*=\mu \cdot X^*$, so the conclusion follows from saturation of the monster model and the fact that $\mu$ is the intersection of all $H^*$'s with $H$ ranging over definable, open subgroups of $G$.\\
(iv) $\rightarrow$ (v) is clear taking $X:=h^{-1}[C]$. For the additional claim, we see that $H^* \cdot X^*$ is invariant under multiplication by $\mu$ on the left and under automorphisms over $M$.\\
(v) $\rightarrow$ (i). 
We already justified in (iv) $\rightarrow$ (v) that $h^{-1}[C]= H^* \cdot X^*$. Since this set is definable, $C$ is clopen.
\end{proof}

\begin{prop}\label{prop: zero-dimensionality}
$S_G^{\mu}(M) \approx G^*/E_\mu$ is zero-dimensional. More precisely, the sets $V_{H,\varphi} := \{ a/E_\mu : a/E_\mu \subseteq H^* \cdot \varphi(G^*)\}$, for a definable, clopen subgroup $H \leq G$ and a formula $\varphi(x)$ over $M$, form a basis of the topology which consists of clopen subsets.
\end{prop}

\begin{proof}
Recall that a basis of open neighborhoods of any element $a/E_\mu \in G^*/E_\mu$ consists of sets of the from $\{b/E_\mu : b/E_\mu \subseteq \psi(G^*)\}$ with $\psi(x)$ ranging over formulas over $M$ such that $a/E_\mu \subseteq \psi(G^*)$.

Now, we show that the collection of all sets $V_{H,\varphi}$ is a basis of the topology. For any basic open neighborhood $V= \{ b/E_\mu : b/E_\mu \subseteq \psi(G^*)\}$ (where $\psi(x)$ is over $M$) of an element $a/E_\mu \in G^*/E_\mu$ we have $a/E_\mu = \mu \cdot \{b: b \equiv_M a\} \subseteq \psi(G^*)$, so, by assumption and compactness, there is a definable, open subgroup $H \leq G$ and a formula $\varphi(x) \in \tp(a/M)$ such that $H^* \cdot \varphi(G^*) \subseteq \psi(G^*)$. This implies $a/E_\mu \in V_{H,\varphi} \subseteq V$.

By Remark \ref{rem: basic on clopens} (v), we see that $H^* \cdot \varphi(G^*)$ is a union of $E_\mu$-classes, so $V_{H,\varphi} = (H^* \cdot \varphi(G^*))/E_\mu$ is clopen by Remark \ref{rem: basic on clopens}.
\end{proof}

Let $\nu$ be a left-invariant, Borel probability measure on $S_G^{\mu}(M)$ (in the proofs of the theorems below, it will exist as a measure witnessing the definable topological amenability of $G$).

\begin{dfn}
An element $p \in S_G^{\mu}(M)$ will be called {\em weakly $\nu$-random} if $\nu(U)>0$ for every open $U \ni p$.
\end{dfn}

By compactness of $S_G^\mu(M)$, a weakly $\nu$-random element in $S_G^\mu(M)$ always exists.
Note also that if $p \in S_G^{\mu}(M)$ is weakly $\nu$-random, then for any formula $\varphi(x)$ over $M$ such that $p \subseteq \varphi(G^*)$, one has $\nu(\varphi(G^*)/E_\mu)>0$. By Proposition \ref{prop: zero-dimensionality}, the last property is actually equivalent to the weak $\nu$-randomness of $p$.  

Let ${\mathcal A}$ be the algebra of subsets of $G$ which are preimages of clopen sets under the map $\Phi \colon G \to G^*/E_\mu$ sending $g \in G$ to $g/E_\mu \in G^*/E_\mu$. By Remark \ref{rem: basic on clopens}, all members of ${\mathcal A}$ are definable (over $M$). This algebra will play a fundamental role in the proof of Theorem \ref{thm: main theorem in Section 2} below.



\begin{rem}\label{rem: description of cal A}
For a subset $X$ of $G$ the following conditions are equivalent:\\
i) $X$ belongs to ${\mathcal A}$,\\
ii) $X=H\varphi(G)$ for some definable, open subgroup $H$ of $G$ and some formula $\varphi(x)$ over $M$,\\
iii) $X$ is definable (over $M$) and $X^*=\mu X^*$.
\end{rem}

\begin{proof}
(i) $\rightarrow$ (ii). Suppose $X=\Phi^{-1}[C]$ for a clopen $C$. By Remark \ref{rem: basic on clopens}, $C= (H^*\varphi(G^*))/E_\mu$ for some definable, open subgroup $H$ of $G$ and some formula $\varphi(x)$ over $M$, and $X=\{ g \in G: g \in H^*\varphi(G^*)\}= H\varphi(G)$.\\
(ii) $\rightarrow$ (iii) is obvious: $\mu X^*=\mu H^*\varphi(G^*)=H^*\varphi(G^*)=X^*$.\\
(iii) $\rightarrow$ (i). By Remark \ref{rem: basic on clopens}, $X^*/E_\mu$ is clopen and $\Phi^{-1}[X^*/E_\mu]= \{ g \in G: g \in X^*\}=X$.
\end{proof}


\begin{cor}\label{cor: good properties of cal A}
If $A \in {\mathcal A}$, $B$ is a definable subset of $G$ and $g \in G$, then $AB \in {\mathcal A}$ and $gA \in {\mathcal A}$. In particular, $AA^{-1} \in {\mathcal A}$.
\end{cor}

Now, we will prove Conjecture \ref{con: topological version of the conjecture from KrPi} (under our extra assumption) and then generalize this argument to get Theorem \ref{thm: main theorem in Section 2}.

\begin{thm}\label{thm:  Topological conjecture}
Let $G$ be a topological group possessing a basis of open neighborhoods of the identity consisting of open subgroups. If $G$ is amenable, then ${G^*}^{00}_{\topo}={G^*}^{000}_{\topo}$.
\end{thm}

\begin{proof}
We are in the situation considered above, although now we can assume that the language contains predicates for all subsets of $G$, because by Corollaries \ref{cor: independence of the language}, \ref{cor: formula for G^00_top} and \ref{cor: independence of the language of G000} neither the universal $G$-ambit  $S_G^{\mu}(M)$ nor the components ${G^*}^{00}_{\topo}$ and  ${G^*}^{000}_{\topo}$ depend on the choice of the language as long as the language contains predicates for all open subsets.
So now every subset of $G$ is automatically definable.

Define a measure ${\mathfrak m}$ on the algebra ${\mathcal A}$ by 
$${\mathfrak m}(\Phi^{-1}[C])= \nu(C),$$ 
where $C$ ranges over clopen subsets of $G^*/E_\mu$ and $\nu$ is the chosen $G$-invariant measure on $G^*/E_\mu$ witnessing amenability. It is easy to check (using the density of the image of $G$ under the function $\Phi$) that ${\mathfrak m}$ is a well-defined, $G$-invariant, finitely additive probability measure on ${\mathcal A}$.

Take a weakly $\nu$-random $q \in S_G^\mu(M)$, and consider the collection $\Sigma(x)$ of all formulas $\varphi(x)$ over $M$ which satisfy $\mu \cdot \varphi(G^*)=\varphi(G^*)$ and are implied by $q$; in other words, by Remark \ref{rem: basic on clopens}, $\varphi(G^*)/E_\mu$ (for $\varphi(x) \in \Sigma(x)$) ranges over all clopen neighborhoods of $q$. By Proposition \ref{prop: zero-dimensionality}, $\bigcap \Sigma(G^*)=q$. 

As in the proof of Proposition \ref{prop: extremely amenable}, we see that $(qq^{-1})^4 \subseteq {G^*}^{000}_{\topo}$. Thus, it remains to prove that  ${G^*}^{00}_{\topo} \subseteq (qq^{-1})^4$. By the previous paragraph, $(qq^{-1})^4= \bigcap_{\varphi(x) \in \Sigma(x)} (\varphi(G^*) \varphi(G^*)^{-1})^4$. Since $q$ is weakly $\nu$-random, for any $\varphi(x) \in \Sigma(x)$ we have ${\mathfrak m}(\varphi(G))>0$. So, by Remark \ref{rem: description of cal A} and Proposition \ref{prop: characterization of G^00_top}, the whole proof boils down to showing the following counterpart of Lemma \ref{lem: main lemma in Section 2}. 

\begin{lem}\label{lem: New version of Lemma 0.3}
Suppose $B \in {\mathcal A}$ satisfies ${\mathfrak m}(B)>0$. Then there are symmetric, generic subsets $C_1 \supseteq C_2 \supseteq ...$ of $G$ such that $C_1 \subseteq (BB^{-1})^4$, $C_{i+1}^2 \subseteq C_i$ and $C_i \in {\mathcal A}$ for all $i$.
\end{lem}

Using Corollary \ref{cor: good properties of cal A}, we easily get that $BB^{-1} \in {\mathcal A}$ is generic and symmetric. Hence, as in the proof of Lemma \ref{lem: main lemma in Section 2}, in order to finish the proof, it is enough to show the following

\begin{clm*}\label{clm: the main claim}
For any $A \in {\mathcal A}$ which is generic and symmetric there is a generic and symmetric set $X \in {\mathcal A}$ such that $X^8 \subseteq A^4$.
\end{clm*}

\begin{proof}
We follow the lines of the proof of Claim \ref{clm: from Frank} in Case 1 (in the proof of Lemma \ref{lem: main lemma in Section 2}), but always working with sets from ${\mathcal A}$ (in particular, $B'$ is now chosen from ${\mathcal A}$). Note that all the sets whose measure is computed during this argument are indeed in ${\mathcal A}$ by virtue of Corollary \ref{cor: good properties of cal A}. So, the only thing to show is that the set $X$ from that proof also belongs to ${\mathcal A}$. 

Recall that $X$ is definable for free. 
Take the monster model considered in Construction $(*)$ and compute things there.
By Remark \ref{rem: description of cal A}, in order to show that $X \in {\mathcal A}$, it is enough to prove that $\mu X^* =X^*$. 
We have $X^* = \{ g \in (A^2)^* : {\mathfrak m}^*(gB'^* \cap B'^*) \in [(t^2/2K){\mathfrak m}(A),{\mathfrak m}(A)]^*\}$.
Consider any $a \in \mu$. Since $B' \in {\mathcal A}$, we have that $aB'^*=B'^*$. So, by the $G^*$-invariance of ${\mathfrak m}^*$, for any $g$ we have ${\mathfrak m}^*(agB'^* \cap B'^*)={\mathfrak m}^*(agB'^* \cap aB'^*)={\mathfrak m}^*(gB'^* \cap B'^*)$. Moreover, since $A \in {\mathcal A}$, we get $A^2 \in {\mathcal A}$, so $a(A^2)^*=(A^2)^*$. The last two conclusions imply that if $g \in X^*$, then $ag \in X^*$, which is enough.
\end{proof}
The proof of Theorem \ref{thm:  Topological conjecture} has been completed.
\end{proof}

We finish this subsection with a proof of Theorem \ref{thm: main theorem in Section 2}.

\begin{proof}[Proof of Theorem \ref{thm: main theorem in Section 2}]

We start as in the proof of Theorem \ref{thm:  Topological conjecture} (except that now not all sets are definable), reducing the proof to Lemma \ref{lem: New version of Lemma 0.3} and then to Claim \ref{clm: the main claim}. However, now not all subsets of $G$ are definable, so we have to apply the full argument from \cite{MaWa} involving the $P_n$'s and $X_n$'s, and there is a technical problem to obtain the desired set $X$ in the algebra  ${\mathcal A}$ (see the proof of the subclaim below). A trick which resolves it is to work with the subalgebra ${\mathcal B}$ of ${\mathcal A}$ defined as the collection of all sets of the form $H_1 \varphi(G) H_2$, where $H_1,H_2$ are definable, open subgroups of G and $\varphi(x)$ is a formula over $M$.
%
%
Equivalently, we can take here $H_1=H_2$, or we can say that ${\mathcal B}$ consists of definable subsets of $G$ which are closed under multiplication on the left and on the right by elements of some open subgroup. Another equivalent definition says that ${\mathcal B}$ is the collection of all definable subsets $X$ of $G$ for which $X^*=\mu X^* \mu$. (Note that here we use the assumption that there is a basis of open neighborhoods of the identity consisting of definable, open subgroups.)

We show the following stronger version of Lemma \ref{lem: New version of Lemma 0.3}.

\begin{lem}\label{lem: Stronger version of Lemma 0.3}
Suppose $B \in {\mathcal A}$ satisfies ${\mathfrak m}(B)>0$. Then there are symmetric, generic subsets $C_1 \supseteq C_2 \supseteq ...$ of $G$ such that $C_1 \subseteq (BB^{-1})^4$, $C_{i+1}^2 \subseteq C_i$ and $C_i \in {\mathcal B}$ for all $i$.
\end{lem}

Since $BB^{-1}$ is generic and symmetric and clearly belongs to ${\mathcal B}$, it is enough to show the following variant of Claim \ref{clm: the main claim}.

\begin{clm*}\label{clm: a variant of the main claim}
For any $A \in {\mathcal B}$ which is generic and symmetric there is a generic and symmetric set $X \in {\mathcal B}$ such that $X^8 \subseteq A^4$.
\end{clm*}

\begin{proof}
Now, we treat ${\mathfrak m}$ as a measure defined only on the algebra ${\mathcal B}$. We take the monster model considered in Construction $(*)$ before Theorem \ref{thm: conjecture from KrPi} (so here ${\mathcal B}$ plays the role of ${\mathcal A}$ in this construction). As in this construction, from now on, by ${\mathfrak m}$ we also denote the extension  $\st \circ \,{\mathfrak m}^*$ of ${\mathfrak m}$ to an invariant measure on the algebra ${\mathcal B}^*$.

We follow the lines of the proof of \cite[Theorem 12]{MaWa} (where we take $A$ there to be our $A^*$ and work in the monster model $M^*$), 
but always working with sets from ${\mathcal B}^*$. In particular, on line 10 of page 61 in \cite{MaWa} we choose $B\in {\mathcal B}^*$ (which
we now call $B'$) satisfying the appropriate requirements.  
Note that all sets whose measure is computed in the course of the proof in \cite{MaWa} are indeed in ${\mathcal B}^*$, because ${\mathcal B}^*$ is closed under products (if $X,Y \in {\mathcal B}^*$, then $XY \in {\mathcal B}^*$) and under multiplication by elements of $G^*$ (if $X \in {\mathcal B}^*$ and $g \in G^*$, then $gX \in {\mathcal B}^*$) which follows from Property (5) of Construction $(*)$ and the fact that ${\mathcal B}$ has such properties (which is obvious). 

Let us recall the definition of the conditions $P_n^t(C)$ and sets $X_n^t(C)$ from \cite{MaWa}, for any $C \subseteq G^*$, $n \in \omega$, and $t \in (0,1]$. Let $K$ be the number of translates of $A$ needed to cover $G$. 
\begin{itemize}
\item $P_0^t(C)$ if  $C \ne \emptyset$.
\item $P_{n+1}^{t}(C)$ if $P_n^t (C)$ holds and $A^*$ is covered by $\floor{\frac{2K}{t}}$ translates of the set
$X_{n}^t(C) := \left\{g \in {A^*}^2 : P_n^{t^2/2K}(gC \cap C) \;\; \textrm{and} \;\; P_n^{t^2/2K}(g^{-1}C \cap C)\right\}$.
\end{itemize}

So the proof from \cite{MaWa} produces a set  $X_n:=X_n^t(B')\subseteq G^*$ (for some $n$) which is definable in $M^*$, symmetric, generic, and satisfies $X_n^8 \subseteq {A^*}^4$. The problem is that the obtained set $X_n=X_n^t(B')$ need not be definable over $M$ and we do not know whether it belongs to the algebra ${\mathcal B^*}$. But we will modify it, to get what we need.

Take a formula $\varphi(x,y)$ without parameters and a tuple $a$ from $M^*$ such that $B'=\varphi(G^*,a)$. For any $b$, let $B'_b=\varphi(G^*,b)$ and $X_{n,b}=X_n^t(B'_b)$. We may assume that $A$ is $\emptyset$-definable. From the definition of $X_n^t$, we easily conclude that there is a formula $\psi(x,y)$ without parameters such that
$X_{n,b}=\psi(G^*,b)$ for any $b$. Let $L$ be the number of translates of $X_n=X_{n,a}$ needed to cover $G^*$. Consider the following conditions on $y$:

\begin{itemize}
\item $X_{n,y}$ is symmetric,
\item $L$ translates of $X_{n,y}$ cover $G^*$,
\item $X_{n,y}^8 \subseteq {A^*}^4$,
\item $B'_y \in {\mathcal B}^*$.
\end{itemize}

By the last paragraph, the first three conditions are definable over $M$ (in the original language of the structure $M$). The last condition is definable over the model $M'$ from Construction $(*)$, but in the expanded language considered in Construction $(*)$, namely, it is defined by the formula ${\mathfrak m}_\varphi(y)<\infty$. Since all these conditions are satisfied by $y:=a$, we can find $a' \in M$ which also satisfies all of them.

From now on, replace $B'=B'_a$ by $B'_{a'}$ and $X_n=X_{n,a}$ by $X_{n,a'}$. Then, $X_n$ is still symmetric and generic, and satisfies $X_n^8 \subseteq {A^*}^4$; moreover, $B' \in {\mathcal B}^*$. But now $X_n$ is definable over $M$ by the formula $\psi(x,a')$, so it makes sense to consider $X_n(G):=\psi(G,a')$, which is obviously  symmetric and generic, and satisfies $X_n(G)^8 \subseteq A^4$. Thus, the only thing to show is that $X_n(G) \in {\mathcal B}$. In order to do that, we first show by induction on $n$ the following

\begin{sbclm}
For any $C$, for each $t$, for all $a \in \mu$ we have $P_n^t(C) \iff P_n^t(aC)$.
\end{sbclm}

\begin{proof}
For $n=0$, it is clear. Suppose it holds for $n$. Take $a \in \mu$. By induction hypothesis and the fact that $A^*$ and so $(A^2)^*$ are both invariant under left and right multiplication by $\mu$ (which follows from the assumption that $A \in {\mathcal B}$), we get 
$$X_{n+1}^t(aC)= a X_{n+1}^t(C) a^{-1}.$$ 
Hence, $A^*$ is covered by the appropriate number of translates of $X_{n+1}^t(C)$ (see line $- 6$ on page 60 in \cite{MaWa}) if and only if it is covered by the same number of translates of $X_{n+1}^t(aC)$ (namely, the translates by the conjugates by $a^{-1}$ of the translating elements for $X_{n+1}^t(C)$); here, we once again use the assumption that $\mu A^* \mu =A^*$ (i.e. $A \in {\mathcal B}$). Therefore, $P_{n+1}^t(C)$ holds if and only if $P_{n+1}^t(aC)$ holds, and the subclaim is proved.
\end{proof}

Since $B' \in {\mathcal B}^*$ and $B'$ is definable over $M$ in the original language, we get by Property (2) of Construction $(*)$ that $B'(G) \in {\mathcal B}$, i.e. $\mu B' \mu = B'$. 

Since $\mu B'=B'$ and $\mu (A^2)^* \mu=(A^2)^*$, using the above subclaim and the definition of $X_n^t$, we easily get that $X_n=X_n^t(B')$ satisfies $\mu X_n \mu=X_n$, so $X_n(G) \in {\mathcal B}$. 
\end{proof}

The proof of Theorem \ref{thm: main theorem in Section 2} has been completed.
\end{proof}

\subsection{Comments and questions}

The assumption on the existence of a basis of open neighborhoods of the identity consisting of open subgroups in Theorems \ref{thm:  Topological conjecture} and \ref{thm: main theorem in Section 2} was needed to have: 
\begin{itemize}
\item Proposition \ref{prop: zero-dimensionality}, i.e. zero-dimensionality of the space $S_G^\mu(M)$, in order to define a measure ${\mathfrak m}$ on the sufficiently rich algebra ${\mathcal A}$ of definable subsets of $G$,
\item to have good properties of ${\mathcal A}$ (see Remarks \ref{rem: basic on clopens} and \ref{rem: description of cal A}).
\end{itemize}
A question is whether one can construct a useful measure without this assumption. 

In the discussion below, we use $E,E',E_1,E_1'$ introduced around Remark \ref{rem: description of E_1'}.
Let us comment on Conjecture \ref{con: the most general}. By Remark \ref{rem: equality of various components}, in the discrete case this conjecture specializes to 

\begin{conj}\label{con: the most general but discrete}
Let $G$ be a group definable in an arbitrary structure $M$. If $G$ is weakly definably amenable, then ${G^*}^{00}_{M}={G^*}^{000}_{M}$.
\end{conj} 

Even in this discrete case [or, more generally, when we have a basis of open neighborhoods of the identity consisting of open subgroups], the space $S_G(M)/E$ [or $S_G(M)/E_1$, respectively] need not be zero-dimensional, so we do not know how to define a measure on a suitable algebra of subsets of $G$. But here also another problem appears. Note that in the proofs of Proposition \ref{prop: extremely amenable}, and Theorems \ref{thm: conjecture from KrPi}, \ref{thm:  Topological conjecture} and \ref{thm: main theorem in Section 2}, the inclusion $pp^{-1} \subseteq {G^*}^{000}_{\defi, \topo}$ for $p \in S_G^\mu(M)$ was almost immediate (and did not require the extra assumption that a basis consists of open subgroups). This property was essential for our proofs to work. However, in the weakly definably [topologically] amenable situation, it is not clear whether this holds. The problem is that we do not know whether $aE'b$ implies $ab^{-1} \in {G^*}^{000}_M$ (see Problem 3.11 in \cite{Kr}). So to prove Conjecture \ref{con: the most general but discrete}, first we would have to show such an implication. If it turned out to be false in general, then one could define a new connected component, say ${G^*}^{000+}_M$, as the normal subgroup generated by all products $ab^{-1}$ for $(a,b) \in E'$. It is $M$-invariant, and by Proposition 3.10 of \cite{Kr}, we see that ${G^*}^{000}_M \leq {G^*}^{000+}_M\leq {G^*}^{00}_M$. But now $pp^{-1} \subseteq  {G^*}^{000+}_M$ for all $p \in S_G(M)/E$, so with this new definition of the connected component at least the second obstacle to prove Conjecture \ref{con: the most general but discrete} disappears. The same applies to the topological context -- to Conjecture  \ref{con: the most general} and the new component ${G^*}^{000+}_{\defi, \topo}$ defined as the normal subgroup generated by all products $ab^{-1}$ for $(a,b) \in E_1'$. 

Note that, with the obvious definition of {\em weak definable topological extreme amenability}, the proof of Proposition \ref{prop: extremely amenable} yields

\begin{rem}
Assume $G$ is weakly definably topologically extremely amenable. Then $G^*={G^*}^{000+}_{\defi, \topo}= {G^*}^{00}_{\defi, \topo}$.  But we do not know whether ${G^*}^{000}_{\defi, \topo}=G^*$, even in the discrete case.
\end{rem}

\section{Amenability and G-compactness}\label{section:amenability of Aut(M)}

The main goal of this section is to prove Theorem \ref{thm: the main result of the paper} and Corollary \ref{theorem: new theorem}. In the first subsection, we recall basic issues around G-compactness. The strategy of our proof of Theorem \ref{thm: the main result of the paper} was described in the introduction. In Subsection \ref{subsection: main subsection of the paper}, we will find the appropriate isomorphisms $\rho$ and $\theta$ which were mentioned in  the discussion following the statement of Theorem \ref{thm: the main result of the paper} in the introduction. In the course of the proof, we will show that the (topological) Bohr compactification of $G:=\aut(M)$ (where $M$ is the $\omega$-categorical structure in question) is exactly 
the natural map $r_1 \colon G \to \gal_{KP}(T)$ and that it is onto; the fact that the Bohr compactification is onto was also proved by Ben-Yaacov in \cite{Be}.

\subsection{Preliminaries around G-compactness}\label{subsection: preliminaries on G-compactness}

We recall some basic and well-known facts on strong types and Galois groups. For more information the reader is referred to \cite{LaPi, CLPZ,GiNe}.

Let $\C$ be a monster model of an arbitrary theory $T$.
\begin{itemize}
		\item $E_L$ is the finest bounded, invariant equivalence relation on a given product of sorts, and its classes are called {\em Lascar strong types},
		\item $E_{KP}$ is the finest bounded, $\emptyset$-type-definable equivalence relation on a given product of sorts, and its classes are called {\em Kim-Pillay strong types}.
               \item $E_{Sh}$ is the intersection of all finite, $\emptyset$-definable equivalence relations on a given product of sorts, and its classes are called {\em Shelah strong types}.
\end{itemize}
Clearly $E_{L}\subseteq E_{KP}\subseteq E_{Sh}$. Then $\autf_L(\C)$, $\autf_{KP}(T)$, and $\autf_{Sh}(T)$ are defined as the groups of all automorphisms of $\C$ preserving all Lascar, Kim-Pillay, and Shelah strong types, respectively, and they are called the {\em groups of Lascar, Kim-Pillay}, and {\em Shelah strong automorphisms}, respectively. It is well-known that: $\autf_L(\C)$ is the subgroup of $\aut(\C)$ generated by all automorphisms fixing small submodels of $\C$ pointwise, i.e. $\autf_L(\C)=\langle \sigma : \sigma \in \aut(\C/M)\;\, \mbox{for some}\;\, M\prec \C\rangle$; $\autf_{KP}(\C)=\aut(\C/\bdd^{heq}(\emptyset))$ (for the definition of the hyperimaginary bounded closure see \cite{Wa}); $\autf_{Sh}(\C)=\aut(\C/\acl^{eq}(\emptyset))$. Then, $\autf_L(\C) \leq \autf_{KP}(T) \leq \autf_{Sh}(T)$ are all normal subgroups of $\aut(\C)$, and the corresponding quotients $\aut(\C)/\autf_L(\C)$, $\aut(\C)/\autf_{KP}(\C)$, and $\aut(\C)/\autf_{Sh}(\C)$ are called {\em Lascar, Kim-Pillay}, and {\em Shelah Galois groups of $T$}, respectively, and they are denoted by $\gal_{L}(T)$, $\gal_{KP}(T)$, and $\gal_{Sh}(T)$. So there are obvious 
group epimorphisms
\begin{figure}[H]
		\centering
		\begin{tikzcd}
			\gal_L(T) \arrow[r, two heads, "h"]&\gal_{KP}(T)\arrow[r, two heads, "g"] &\gal_{Sh}(T).
		\end{tikzcd}
	\end{figure}

\begin{fct}\label{fct: Galois groups are absolute}
The above Galois groups do not depend (up to isomorphism) on the choice of the monster model $\C$; for example, for $\C \prec \C'$, the map taking $\sigma/\autf_L(\C)$ to $\sigma'/\autf_L(\C')$,  for any extension $\sigma' \in \aut(\C')$ of $\sigma \in \aut(\C)$, is a well-defined group isomorphism.
\end{fct}

\begin{dfn}
i) The theory $T$ is {\em G-compact} if $h$ is an isomorphism. Equivalently, if $\autf_{L}(\C)=\autf_{KP}(\C)$.\\
ii)  The theory $T$ is {\em G-trivial} if $\gal(T_A)$ is trivial for any finite set $A \subseteq \C$, where $T_A$ is the elementary diagram of $A$ (i.e. the theory of $\C$ in the language expanded by constants from $A$). 
\end{dfn}

The relations $E_L$, $E_{KP}$, and $E_{Sh}$ turn out to be the orbit equivalence relations of $\autf_L(\C)$, $\autf_{KP}(\C)$, and $\autf_{Sh}(\C)$, respectively, which implies that $T$ is G-compact if and only if $E_L=E_{KP}$ on all (also infinite) tuples.

Now, we recall the logic topology on $\gal_L(T)$. For more details consult \cite{LaPi} and \cite{GiNe}.
	Let $\nu\fcolon \aut(\C) \to \gal_L(\C)$ be the quotient map. Choose a small model $M$, and let $\bar m$ be its enumeration. By $S_{\bar m}(M)$ we denote $\{ \tp(\bar n/M): \bar n \equiv \bar m\}$. Let $\nu_1\fcolon \aut(\C) \to S_{\bar m}(M)$ be defined by $\nu_1(\sigma)=\tp(\sigma(\bar m)/M)$, and $\nu_2\fcolon S_{\bar m}(M) \to \gal_L(T)$ by $\nu_2(\tp(\sigma(\bar m)/M))=\sigma /\autf_L(\C)$. Then $\nu_2$ is a well-defined surjection, and $\nu=\nu_2 \circ \nu_1$. Thus, $\gal_L(T)$ becomes the quotient of the space $S_{\bar m}(M)$ by the relation of lying in the same fiber of $\nu_2$, and so we can define a topology on $\gal_L(T)$ as the quotient topology. In this way, $\gal_L(T)$ becomes a quasi-compact (so not necessarily Hausdorff) topological group. This topology does not depend on the choice of the model $M$.
	
	\begin{fct}\label{fct: characterization of topology on Gal_L(T)}
		The following conditions are equivalent for $C \subseteq \gal_L(T)$.
		\begin{enumerate}[label=\roman{*}),nosep,]
			\item
			$C$ closed.
\item For every tuple $\bar m$ enumerating a small submodel of $\C$ there is a partial type $\pi(\bar x)$ (with parameters) such that $\nu^{-1}[C]=\{ \sigma \in \aut(\C) : \sigma(\bar m) \models \pi(\bar x)\}$.
			\item
			There are a tuple $\bar a$ and a partial type $\pi(\bar x)$ (with parameters) such that $\nu^{-1}[C]=\{ \sigma \in \aut(\C) : \sigma(\bar a) \models \pi(\bar x)\}$.
		\end{enumerate}
	\end{fct}

The {\em logic topologies} on $\gal_{KP}(T)$ and $\gal_{Sh}(T)$ are the quotient topologies coming from the logic topology on $\gal_L(T)$ and epimorphisms $h$ and $g$. Fact \ref{fct: Galois groups are absolute} also works for the Galois groups treated as topological groups.

The group $\gal_0(T)$ is defined as the closure of the identity in $\gal_L(T)$. It turns out that $\gal_0(T)=\autf_{KP}(\C)/\autf_L(\C)$ and $\gal_{KP}(T) \cong \gal_{L}(T)/\gal_0(T)$, so $\gal_{KP}(T)$ is a compact (Hausdorff) group; and so is $\gal_{Sh}(T)$ (it is even profinite).

The following was proved by Kim in \cite{Ki} for finite tuples; it extends to arbitrary tuples by compactness.
(In the $\omega$-categorical case, Kim's result is immediate, since a $\emptyset$-type-definable equivalence relation is $\emptyset$-definable, and so if it is bounded, it must be finite.)

\begin{fct}\label{fct: Kim's theorem}
If $T$ is $\omega$-categorical (or, more generally, small), then $E_{KP}=E_{Sh}$. Thus, $\autf_{KP}(\C)=\autf_{Sh}(\C)$.
\end{fct}

\subsection{Groups of automorphisms of $\omega$-categorical structures}\label{subsection: main subsection of the paper}

Throughout this subsection, $M$ is a countable,  $\omega$-categorical structure. Without loss of generality $M=M^{eq}$. Let $T=\Th(M)$. By $G$ we denote the group $\aut(M)$ of automorphisms of $M$, and we treat it as a topological group equipped with the pointwise convergence topology. Being a closed subgroup of $S_\infty$, it is a Polish group. In fact, all the closed subgroups of $S_\infty$ are precisely the groups of automorphisms of countable first order structures, and this class coincides with the class of all Polish groups possessing a (countable) basis of open neighborhoods of the identity consisting of open subgroups (see \cite[Theorem 1.5.1]{BeKe}).

Let $r_1 \colon \aut(M) \to \gal_{KP}(T)$ be given by $r_1(\sigma)=\sigma'/\autf_{KP}(\C)$ for any $\sigma' \in \aut(\C)$ extending $\sigma$, where $\C\succ M $ is a monster model. It is well-defined by the fact that automorphisms fixing a model are Lascar strong, and it is also a homomorphism. 

\begin{thm}\label{thm: r_1 is the Bohr compactification}
The function $r_1$ is surjective and it is the Bohr compactification of $G$.
\end{thm}

From a series of remarks and lemmas we will conclude this theorem (see Corollary \ref{cor: proof of theorem 4.5}), and, more importantly, we will find the desired isomorphism $\theta$ from the diagram in the introduction. Then we will use a similar method to find $\rho$.

\begin{rem}\label{rem: adding acl eq}
Let $T_{\acl^{eq}(\emptyset)}$ be the elementary diagram of $\acl^{eq}(\emptyset)$. Then $T_{\acl^{eq}(\emptyset)}$ is $\omega$-categorical.
\end{rem}

\begin{proof}
$E_{Sh}$ on tuples of length $n \in \omega$ is invariant over $\emptyset$, so it is $\emptyset$-definable. But it is bounded, so it has finitely many classes. Hence, for every $n \in \omega$ there are only finitely many types in $S_n(\acl^{eq}(\emptyset))$. Thus, $T_{\acl^{eq}(\emptyset)}$ is $\omega$-categorical.
\end{proof}



By Fact \ref{fct: Kim's theorem}, $\autf_{KP}(\C)=\autf_{Sh}(\C)$, so $\gal_{KP}(T)=\gal_{Sh}(T)$ which can be naturally identified with the group of all elementary permutations of $\acl^{eq}(\emptyset)$. 

\begin{rem}\label{rem: r_1 is onto}
$r_1$ is onto.
\end{rem}

\begin{proof}
By the above comment, it is enough to show that every elementary permutation of $\acl^{eq}(\emptyset)$ can be realized by an automorphism of $M$. But this follows immediately from Remark \ref{rem: adding acl eq}.
\end{proof}


\begin{rem}\label{rem: r_1 is continuous}
$r_1$ is continuous.
\end{rem}

\begin{proof}
Let $D \subseteq \gal_{KP}(T)$ be closed. By Fact \ref{fct: characterization of topology on Gal_L(T)}, $r_1^{-1}[D]=\{ \sigma \in \aut(M) : \sigma(\bar m) \models \pi(\bar x)\}$, where $\bar m$ is an enumeration of $M$ and $\pi(\bar x)$ is a partial type, and we see that this set is closed in the pointwise convergence topology.
\end{proof}


Consider the new structure ${\mathcal M}$ consisting of the structure $M$ together with the group $G=\aut(M)$ acting on $M$, expanded by predicates 
for all open subsets of $G$.
For convenience (to avoid some density arguments) we will work with two monster models 
$${\mathcal M}^{**}=(M^{**},G^{**},\dots) \succ {\mathcal M}^*=(M^*,G^*,\dots) \succ {\mathcal M}=(M,G,\dots)$$ 
such that ${\mathcal M}^{**}$ is a monster model with respect to ${\mathcal M}^*$, and $M^{**} \succ M^* \succ M$ are monster models of the original theory. Let the above $\C$ be equal to $M^*$; $\mu$ will denote the infinitesimals of $G$ computed in ${\mathcal M}^{**}$.

It is clear that each element $\sigma$ of $G^*$ [or of $G^{**}$] induces an automorphism $\overline{\sigma} \in \aut(M^*)$ [or $\overline{\sigma} \in \aut(M^{**})$, respectively], and it is determined by this automorphism. But not every automorphism of $M^*$ [or of $M^{**}$] arises in this way. 

When a group $F$ acts on a space $X$ and $A\subseteq X$, by $\Fix_F(A)$ we will denote the pointwise stabilizer of $A$ in $F$. The next remark is obvious.

\begin{rem}\label{rem: mu as stabilizer}
$\mu=\bigcap \{ \Fix_G(\bar a)^{**}: \bar a \;\, \mbox{finite tuple in}\; M\}= \Fix_{G^{**}}(M)$.
\end{rem}

Let $H$ be the group of all $\sigma \in G^{**}$ which fix $M^*$ setwise (hence induce an automorphism of $M^*$). Then $H \leq G^{**}$. Put $\widetilde{H}:=\{ \sigma |_{M^*}: \sigma \in H\} \leq \aut(M^*)$.

\begin{lem}\label{lem: H gives aut(M^*)}
$\widetilde{H}=\aut(M^*)$.
\end{lem}

\begin{proof}
Take any $\sigma \in \aut(M^*)$. We need to show that $\sigma \in \widetilde{H}$. By $|M^*|^{+}$-saturation of ${\mathcal M}^{**}$, it is enough to show that for any finite tuple $\bar a$ from $M^*$ there is $\tau \in G^{**}$ such that $\sigma(\bar a)=\tau(\bar a)$. So consider such a tuple $\bar a$ of length $n$.

By $\omega$-categoricity, the condition $\bar x \equiv_\emptyset \bar y$ is $\emptyset$-definable in the original theory and
$${\mathcal M} \models (\forall \bar x, \bar y \in M^n)(\bar x \equiv_{\emptyset} \bar y \rightarrow (\exists f \in G)( f(\bar x)=\bar y)).$$
So the same holds in ${\mathcal M}^*$, so we have
$$ (\forall \bar x, \bar y \in {M^*}^n)(\bar x \equiv_{\emptyset} \bar y \rightarrow (\exists f \in G^*)( f(\bar x)=\bar y)).$$
Since $\bar a \equiv_{\emptyset} \sigma(\bar a)$, we conclude that there is $\tau \in G^* \leq G^{**}$ such that $\sigma(\bar a)=\tau(\bar a)$.
\end{proof}

Put $\widetilde{H}^{00}_{\topo}:=\{ \sigma |_{M^*} : \sigma \in H \cap {G^{**}}^{00}_{\topo}\} \leq \aut(M^*)$. We would like to stress that $\widetilde{H}^{00}_{\topo}$ is a local notation which should not be confused with the notation from Section \ref{section: dynamics of top groups}. (Note that $\widetilde{H}=\aut(M^*)$ is a topological group, but it is not saturated).

\begin{lem}\label{lem: closedness and boundedness}
$\widetilde{H}^{00}_{\topo}$ is a closed (in the pointwise convergence topology on $\aut(M^*)$), normal, bounded index (i.e. smaller than the degree of saturation of $M^*$) subgroup of $\aut(M^*)$.
\end{lem}

\begin{proof}
Normality follows from Lemma \ref{lem: H gives aut(M^*)} and the normality of ${G^{**}}^{00}_{\topo}$ in $G^{**}$. Since ${G^{**}}^{00}_{\topo}$ has bounded index in $G^{**}$, this index is at most $2^{|T'|}$, where $T'$ is the theory of ${\mathcal M}$, which in turn is smaller that the degree of saturation of $M^*$. Thus, the index of $\widetilde{H}^{00}_{\topo}$ in $\aut(M^*)$ is bounded by Lemma  \ref{lem: H gives aut(M^*)}.

It remains to check closedness. Consider any $\sigma \in \aut(M^*) \setminus \widetilde{H}^{00}_{\topo}$. We will show that $\sigma \notin \cl(\widetilde{H}^{00}_{\topo})$. For this we need to find an open neighborhood of $\sigma$ disjoint from $\widetilde{H}^{00}_{\topo}$. Take $\sigma' \in H$ such that $\sigma' |_{M^*} =\sigma$. Then $\sigma' \notin {G^{**}}^{00}_{\topo}$. So, by Remark \ref{rem: mu as stabilizer}, there is a finite tuple $\bar a$ in $M$ such that $\sigma' \notin {G^{**}}^{00}_{\topo} \cdot \Fix_{G}(\bar a)^{**}$. We claim that $\sigma \cdot \Fix_{\aut(M^*)}(\bar a) \cap \widetilde{H}^{00}_{\topo}=\emptyset$, which clearly completes our proof.

Suppose for a contradiction that $\sigma \tau =\eta$ for some $\tau \in \Fix_{\aut(M^*)}(\bar a)$ and $\eta \in \widetilde{H}^{00}_{\topo}$. By Lemma \ref{lem: H gives aut(M^*)}, $\tau=\tau'|_{M^*}$ for some $\tau' \in H$; then $\tau' \in \Fix_G(\bar a)^{**} \cap H$. We also have $\eta = \eta'|_{M^*}$ for some $\eta' \in  H \cap {G^{**}}^{00}_{\topo}$. Then $(\eta'^{-1}\sigma'\tau')|_{M^*}= \id_{M^*}$, and so, by Remark \ref{rem: mu as stabilizer}, $\eta'^{-1}\sigma'\tau' \in H \cap {G^{**}}^{00}_{\topo}$. We conclude that  $\sigma' \in {G^{**}}^{00}_{\topo} \cdot \Fix_G(\bar a)^{**}$, a contradiction.
\end{proof}

\begin{cor}\label{cor: important inclusion}
$\autf_{Sh}(M^*) \leq \widetilde{H}^{00}_{\topo}$.
\end{cor}

\begin{proof}
By Lemma \ref{lem: closedness and boundedness} (more precisely, by closedness of $\widetilde{H}^{00}_{\topo}$), it is enough to show that $\autf_{Sh}(M^*) \leq \cl(\widetilde{H}^{00}_{\topo})$. For this consider any $f \in \autf_{Sh}(M^*)$ and take any finite tuple $\bar a$ in $M^*$. Then consider the orbit equivalence relation $E$ of $\widetilde{H}^{00}_{\topo}$ on the sort of $\bar a$ in $M^*$. By Lemma \ref{lem: closedness and boundedness}, $E$ is a bounded, invariant equivalence relation. So it is $\emptyset$-definable (by $\omega$-categoricity) and finite. Hence, $f$ fixes $\bar a /E$, so $f(\bar a)=\sigma(\bar a)$ for some $\sigma \in  \widetilde{H}^{00}_{\topo}$.
\end{proof}

Define $\theta' \colon \aut(M^*) \to G^{**}/{G^{**}}^{00}_{\topo}$ by $\theta'(\sigma)=\sigma'/{G^{**}}^{00}_{\topo}$ for some [any] $\sigma' \in H$ such that $\sigma'|_{M^*}=\sigma$. The existence of such a $\sigma'$ is guaranteed by Lemma \ref{lem: H gives aut(M^*)}; the fact that $\sigma'/{G^{**}}^{00}_{\topo}$ does not depend on the choice of $\sigma' \in H$ such that $\sigma'|_{M^*}=\sigma$ follows from Remark \ref{rem: mu as stabilizer} (namely, $\Fix_{G^{**}}(M^*) \leq \mu \leq {G^{**}}^{00}_{\topo}$).

An easy computation shows that $\theta'$ is a group homomorphism. By Corollary \ref{cor: important inclusion}, $\theta'$ factors through $\autf_{Sh}(M^*)$, so we get the induced homomorphism 
$$\theta \colon \aut(M^*)/\autf_{Sh}(M^*) \to G^{**}/{G^{**}}^{00}_{\topo},$$ 
and we will see that this is the isomorphism that we are looking for (note that $G^{**}/{G^{**}}^{00}_{\topo}$ is naturally identified with $G^{*}/{G^{*}}^{00}_{\topo}$). 

Let $r_2 \colon G \to G^{**}/{G^{**}}^{00}_{\topo}$ be the quotient map. 

\begin{lem}
The following diagram commutes.
\begin{figure}[H]
		\centering
		\begin{tikzcd}
			& \gal_{KP}(T)\arrow[dd,"\theta"]\\
		     G\arrow[ur,"r_1"]\arrow[dr,"r_2"] & \\
			& G^{**}/{G^{**}}^{00}_{\topo}. 	
		\end{tikzcd}
	\end{figure}
\end{lem}

\begin{proof}
Take any $\sigma \in G$.  Then $\sigma \in G^{**}$ and $r_2(\sigma)=\sigma /{G^{**}}^{00}_{\topo}$.
Also $\sigma \in G^*$, and let $\bar \sigma$ be $\sigma$ treated as an element of $\aut(M^*)$. Then $r_1(\sigma)=\bar \sigma / \autf_{Sh}(M^*)$.
Finally, we see that $\sigma \in H$ and $\sigma |_{M^*}=\bar \sigma$, so $\theta (\bar \sigma / \autf_{Sh}(M^*))=\sigma /{G^{**}}^{00}_{\topo}$.
\end{proof}

Note that $\ker(r_1)=\autf_{Sh}(M):=\aut(M/\acl^{eq}(\emptyset))$. So, by Remark \ref{rem: r_1 is onto}, $r_1$ induces a group isomorphism $r \colon \aut(M)/\autf_{Sh}(M) \to \gal_{KP}(T)$. So if $q \colon G \to \aut(M)/\autf_{Sh}(M)$ is the quotient map, the following diagram commutes.
\begin{figure}[H]
		\centering
		\begin{tikzcd}
			& \aut(M)/\autf_{Sh}(M)\arrow[dd,two heads,"r"]\\
		     G\arrow[ur,two heads,"q"]\arrow[dr,two heads,"r_1"] & \\
			& \gal_{KP}(T). 	
		\end{tikzcd}
	\end{figure}

\begin{lem}\label{lem: r is a homeomorphism}
$r$ is a homeomorphism (so topological isomorphism). Thus, the topology on $\gal_{KP}(T)$ is the quotient topology induced by $r_1$.
\end{lem}

\begin{proof}
By Remark \ref{rem: r_1 is continuous}, $r_1$ is continuous. So $r$ is continuous with $\aut(M)/\autf_{Sh}(M)$ equipped with the quotient topology. But $\autf_{Sh}(M)$ is a closed subgroup of $\aut(M)$, so $\aut(M)/\autf_{Sh}(M)$ is a Polish group. Since $\gal_{KP}(T)$ is also Polish (because the language is countable as a part of the $\omega$-categoricity assumption), $r$ is a homeomorphism (see \cite[Theorem 1.2.6]{BeKe}). The rest is clear.
\end{proof}

\begin{cor}
$\theta$ is continuous.
\end{cor}

\begin{proof}
By Fact \ref{fac: from GPP}, we know that $r_2$ is the Bohr compactification of $G$, so it is continuous. Thus, we finish using Lemma \ref{lem: r is a homeomorphism} and the first diagram above.
\end{proof}

We finish the discussion of $\theta$ with the following corollary, which completes the proof of Theorem \ref{thm: r_1 is the Bohr compactification}.

\begin{cor}\label{cor: proof of theorem 4.5}
$\theta$ is a topological isomorphism, and the epimorphism $r_1$ is the Bohr compactification of $G$ (in particular, Theorem \ref{thm: r_1 is the Bohr compactification} is true).
\end{cor}

\begin{proof}
This follows from the fact (Fact \ref{fac: from GPP}) that $r_2$ is the Bohr compactification of $G$, $r_1$ is a surjective compactification of $G$, and $\theta$ is a morphism from $r_1$ to $r_2$.
\end{proof}

In order to define the desired $\rho$, first define $\rho' \colon \aut(M^*) \to G^{**}/{G^{**}}^{000}_{\topo}$ by $\rho'(\sigma)=\sigma'/{G^{**}}^{000}_{\topo}$ for some [any] $\sigma' \in H$ such that $\sigma'|_{M^*}=\sigma$. As in the case of $\theta'$, the existence of such a $\sigma'$ is guaranteed by Lemma \ref{lem: H gives aut(M^*)}; the fact that $\sigma'/{G^{**}}^{000}_{\topo}$ does not depend on the choice of $\sigma' \in H$ such that $\sigma'|_{M^*}=\sigma$ follows from Remark \ref{rem: mu as stabilizer}. An easy computation shows that $\theta'$ is a group homomorphism. In order to factorize $\tau'$ through $\autf_L(M^*)$, we need to prove the following counterpart of Lemma \ref{cor: important inclusion}.

\begin{lem}
$\autf_L(M^*) \leq \widetilde{H}^{000}_{\topo}:= \{ \sigma |_{M^*} : \sigma \in H \cap {G^{**}}^{000}_{\topo}\} \leq \aut(M^*)$.
\end{lem}

\begin{proof}
As in the first paragraph of the proof of Lemma \ref{lem: closedness and boundedness}, we see that $\widetilde{H}^{000}_{\topo}$ is a normal, bounded index  subgroup of $\aut(M^*)$.

Consider the orbit equivalence relation $E$ of $\widetilde{H}^{000}_{\topo}$ on the sort of $\bar m$ in $M^*$, where $\bar m$ is an enumeration of $M$. We get that $E$ is bounded and invariant.

Take any $f \in \autf_L(M^*)$. Then $f$ fixes $\bar m/E$, so there is $\sigma \in  \widetilde{H}^{000}_{\topo}$ such that $f(\bar m)=\sigma(\bar m)$, i.e.  $(\sigma^{-1}f)(\bar m)=\bar m$. By Lemma \ref{lem: H gives aut(M^*)}, choose $\sigma',f' \in H$ such that $\sigma'|_{M^*}=\sigma$ and $f'|_{M^*}=f$. Then $(\sigma'^{-1}f') (\bar m)=\bar m$, so $\sigma'^{-1}f' \in H \cap {G^{**}}^{000}_{\topo}$, so $\sigma^{-1}f \in \widetilde{H}^{000}_{\topo}$. Since  $\sigma \in  \widetilde{H}^{000}_{\topo}$, we conclude that $f \in \widetilde{H}^{000}_{\topo}$.
\end{proof}

So, $\rho'$ factors through $\autf_L(M^*)$ and yields a group homomorphism
$$\rho \colon \aut(M^*)/\autf_L(M^*) \to G^{**}/{G^{**}}^{000}_{\topo}.$$
This will be the required $\rho$ (note that $G^{**}/{G^{**}}^{000}_{\topo}$ naturally identifies with $G^{*}/{G^{*}}^{000}_{\topo}$).
From the explicit definitions of $\theta$ and $\rho$, we immediately get that the diagram from the introduction commutes:

\begin{figure}[H]
		\centering
		\begin{tikzcd}
			\aut(M^*)/\autf_L(M^*) \arrow[r,"h"]\arrow[d,"\rho"]&\aut(M^*)/\autf_{Sh}(M^*)\arrow[d,"\theta"] \\
			G^{**}/{G^{**}}^{000}_{\topo}\arrow[r] & G^{**}/{G^{**}}^{00}_{\topo},
		\end{tikzcd}
	\end{figure}

So, in order to finish the proof of Theorem \ref{thm: the main result of the paper}, it remains to show

\begin{lem}\label{lem: rho is an isomorphism}
$\rho$ is a group isomorphism.
\end{lem}

\begin{proof}
We are going to use Proposition \ref{prop: universal description}, or rather the comment right below the proof of this proposition which says that the map 
$$F \colon S_G({\mathcal M}) \to G^{**}/{G^{**}}^{000}_{\topo}$$
given by $F(p)=\sigma/{G^{**}}^{000}_{\topo}$ for any $\sigma \models p$ is the initial object in the category ${\mathcal C}$ of maps $S_G({\mathcal M}) \to L$, where $L$ is a group, induced by homomorphisms $G^{**} \to L$ trivial on $\mu$.

Consider $f \colon G^{**} \to \aut(M^{**})/\autf_L(M^{**})$ given by $f(\sigma)=\bar \sigma/\autf_L(M^{**})$, where $\bar \sigma \in \aut(M^{**})$ is the automorphism of $M^{**}$ induced by $\sigma \in G^{**}$.

\begin{clm*}\label{clm: to be an object of C}
i) If $\sigma, \tau \in G^{**}$ have the same type over $M$, then $f(\sigma)=f(\tau)$.\\
ii) $f[\mu]=\{\id/\autf_L(M^{**})\}$.
\end{clm*}

\begin{proof}
i) By assumption, there is $s \in \aut({\mathcal M}^{**}/M)$ such that $s(\sigma)=\tau$. Then  $s|_{M^{**}} \in \aut(M^{**}/M)$, so $s|_{M^{**}} \in \autf_L(M^{**})$, and 
$$(s|_{M^{**}} \circ \bar \sigma) (\bar m)=s(\bar \sigma(\bar m))=s(\sigma \bar m)= s(\sigma) s(\bar m)=\tau \bar m=\bar \tau (\bar m),$$
where $\bar m$ is an enumeration of $M$. Hence, $\bar \sigma/\autf_L(M^{**}) = \bar \tau / \autf_L(M^{**})$.\\[1mm]
ii) follows from Remark \ref{rem: mu as stabilizer}.
\end{proof}

There is also a group isomorphism 
$$g \colon \aut(M^{*})/\autf_L(M^{*}) \to \aut(M^{**})/\autf_L(M^{**})$$ 
given by $g(\sigma/\autf_L(M^{*}))=\sigma'/\autf_L(M^{**})$ for any $\sigma' \in \aut(M^{**})$ extending $\sigma$.

So,  $g^{-1} \circ f \colon G^{**} \to \aut(M^{*})/\autf_L(M^{*})$ is a group homomorphism, which, by Claim \ref{clm: to be an object of C}, satisfies:

\begin{itemize}
\item if $\sigma, \tau \in G^{**}$ have the same type over $M$, then $(g^{-1} \circ f)(\sigma)=(g^{-1} \circ f)(\tau)$,
\item $(g^{-1} \circ f)[\mu]=\{\id/\autf_L(M^{*})\}$.
\end{itemize}

Therefore, $g^{-1} \circ f$ induces a map $\widehat{g^{-1} \circ f} \colon S_G({\mathcal M}) \to \aut(M^{*})/\autf_L(M^{*})$ which is an object of the category ${\mathcal C}$.

\begin{clm*}\label{clm: to use Prop. 2.8}
The following diagram commutes.
\begin{figure}[H]
		\centering
		\begin{tikzcd}
			& \aut(M^*)/\autf_{L}(M^*)\arrow[dd,"\rho"]\\
		     S_G({\mathcal M})\arrow[ur,"\widehat{g^{-1} \circ f}"]\arrow[dr,"F"] & \\
			& G^{**}/{G^{**}}^{000}_{\topo}. 	
		\end{tikzcd}
	\end{figure}
\end{clm*}

\begin{proof}
Consider $p \in S_G({\mathcal M})$ and take any $\sigma \in G^*\leq G^{**}$ such that $\sigma \models p$. Then $F(p)=\sigma /{G^{**}}^{000}_{\topo}$.
Let $\bar \sigma \in \aut(M^{**})$ be induced by $\sigma$. Then $f(\sigma)=\bar \sigma/ \autf_L(M^{**})$, so $(g^{-1}\circ f)(\sigma)=\bar \sigma |_{M^*}/ \autf_L(M^{*})$. Hence, $\widehat{g^{-1}\circ f}(p)= \bar \sigma |_{M^*}/ \autf_L(M^{*})$. We conclude that 
$$\rho(\widehat{g^{-1}\circ f}(p))=\sigma/{G^{**}}^{000}_{\topo} =F(p).$$
\end{proof}

By Lemma \ref{lem: H gives aut(M^*)}, the function $g^{-1}\circ f$ is easily seen to be onto, so $\widehat{g^{-1} \circ f}$ is also surjective. Using this together with the observation that $\widehat{g^{-1} \circ f}$ is an object of ${\mathcal C}$, the first paragraph of this proof, and Claim \ref{clm: to use Prop. 2.8}, we get that $\rho$ is an isomorphism. 
\end{proof}

The proof of Theorem \ref{thm: the main result of the paper} has been completed. As was mentioned in the introduction, the original definition of G-compactness in \cite{La} was stronger in the sense that naming any finite set of parameters was allowed. Now, we give an explanation that Theorem \ref{thm: the main result of the paper} is true even with this stronger definition.

So, take any finite subset $A$ of $M$ (by $\omega$-categoricity, it is enough to consider parameters from $M$). Then $\aut(M/A)$ is the group of automorphisms of $M$ with constants for members of $A$ added to the language. Since the resulting theory is still $\omega$-categorical and we already have proved Theorem \ref{thm: the main result of the paper} (with the weaker definition of G-compactness), it remains to show that $\aut(M/A)$ is amenable. Since this is an open subgroup of $\aut(M)$, we finish using the following well-known fact (see Theorem 3.3 of \cite{Ri}).

\begin{fct}\label{fac: folklore on amenability}
Let $G$ be a topological group and $H$ an open subgroup. If $G$ is amenable, then so is $H$.
\end{fct}



We finish with the proof of Corollary \ref{theorem: new theorem}.

\begin{proof}[Proof of Corollary \ref{theorem: new theorem}]
If $\aut(M)$ is extremely amenable, Proposition \ref{prop: extremely amenable} together with Lemma \ref{lem: rho is an isomorphism} imply that $\gal_L(T)$ is trivial. Thus, by $\omega$-categoricity, it remains to check that extreme amenability of $\aut(M)$ is preserved under naming finitely many parameters from $M$. So, as in above discussion, it is enough to show the following counterpart of Fact \ref{fac: folklore on amenability}.

\begin{fct}\label{fac: folklore on extreme amenability}
An open subgroup $H$ of an extremely amenable topological group $G$ is extremely amenable.
\end{fct}

\begin{proof}
We will use Pestov's characterization of extreme amenability in terms of generic sets \cite[Theorem 8.1]{Pe}: $G$ is extremely amenable if and only if for any left generic subset $S$ of $G$, $SS^{-1}$ is dense in $G$.

Take any left generic subset $S$ of $H$. We want to show that $SS^{-1}$ is dense in $H$. Choose any set $R$ of representatives of right cosets of $H$ in $G$. Consider $S':=SR$. We see that $S'$ is generic in $G$ (witnessed by left translates by the elements witnessing genericity of $S$ in $H$). So, by the extreme amenability of $G$ and Pestov's characterization, we get that $S'S'^{-1}$ is dense in $G$. But $S'S'^{-1} \cap H=SS^{-1}$ and $H$ is open. Thus, we conclude that $SS^{-1}$ is dense in $H$, and we are done. 
\end{proof}

The proof of Corollary \ref{theorem: new theorem} has been completed.
\end{proof}

\begin{rem}
We give a direct account of Corollary \ref{theorem: new theorem}, which also shows that the Lascar equivalence relation is realized in one step.  First, deduce from extreme amenability of $\aut(M)$ that every complete type over $\emptyset$ has an extension to a complete type over $M$ which is $\aut(M)$-invariant. Now, suppose $b$ and $c$ are finite tuples  (without loss from $M$) with the same type over $\emptyset$. Then for every complete type $q$ over $\emptyset$ there is a realization $m$ of $q$ with $\tp(b/m) = \tp(c/m)$ (choose $m$ to realize the invariant extension of $q$ over $M$). By compactness, we can find a model $M_{0}$ such that $\tp(b/M_{0}) = \tp(c/M_{0})$, whence $b$ and $c$ have the same Lascar strong type over $\emptyset$. By compactness, the observation  that $a\equiv b$ implies $a\equiv_{M_0} b$ for some model $M_0$ (and so $a$ and $b$ have the same Lascar strong type over $\emptyset$) also holds for all infinite tuples $a,b$. Therefore, $\gal_L(T)$ is trivial. Using Fact 4.20, we get that the same is true over any finite set of parameters, so $T$ is G-trivial. 
\end{rem}

\section*{Acknowledgments}
We would like to thank Udi Hrushovski for helpful remarks and comments on an earlier draft of the paper, and to the anonymous referee for careful reading, correcting typos, and suggesting some references.


\end{document}